\documentclass[a4paper, 12pt]{amsart}
\usepackage{amscd,amssymb,epsfig,amsxtra,amsfonts}
\usepackage{amsmath}
\usepackage{stmaryrd, mathrsfs}
\usepackage{verbatim}
\usepackage{epsfig} 
\newcommand{\ms}[1]{\mbox{\tiny$#1$}}
\newcommand{\epsh}[2]
          {\begin{array}{c} \hspace{-1.3mm}
         \raisebox{-4pt}{\epsfig{figure=#1,height=#2}}
         \hspace{-1.9mm}\end{array}}

\newcommand{\lk}{\operatorname{lk}}
\newcommand{\coev}{\stackrel{\longrightarrow}{\operatorname{coev}}}
\newcommand{\ev}{\stackrel{\longrightarrow}{\operatorname{ev}}}
\newcommand{\tev}{\stackrel{\longleftarrow}{\operatorname{ev}}}
\newcommand{\tcoev}{\stackrel{\longleftarrow}{\operatorname{coev}}}

\newcommand{\Id}{\operatorname{Id}}
\newcommand{\qdim}{\operatorname{qdim}}
\newcommand{\Ker}{\operatorname{Ker}}
\newcommand{\Alg}{\operatorname{Alg}}
\newcommand{\End}{\operatorname{End}}
\newcommand{\Hom}{\operatorname{Hom}}

\newcommand{\tr}{\operatorname{tr}}
\newcommand{\ptr}{\operatorname{ptr}}
\newtheorem{theo}{Theorem}[section]
\newtheorem{Def}[theo]{Definition}
\newtheorem{The}[theo]{Theorem}
\newtheorem{Lem}[theo]{Lemma}

\newtheorem{Pro}[theo]{Proposition}
\newtheorem{HQ}[theo]{Corollary}
\newtheorem{rmq}[theo]{Remark}
\begin{document}
\title[Topological invariants from $\mathcal{U}_{\xi}\mathfrak{sl}(2|1)$]{Topological invariants from quantum group $\mathcal{U}_{\xi}\mathfrak{sl}(2|1)$ at roots of unity}


\author{Ngoc-Phu HA}        
\address{Laboratoire de Math\'ematiques de Bretagne Atlantique, Universit\'e de Bretagne Sud, BP 573, 56017 Vannes, France}
\email{ngoc-phu.ha@univ-ubs.fr}

\maketitle

\begin{abstract}
In this article we construct link invariants and 3-manifold invariants from the quantum group associated with Lie superalgebra $\mathfrak{sl}(2|1)$. This construction based on nilpotent irreducible finite dimensional representations of quantum group $\mathcal{U}_{\xi}\mathfrak{sl}(2|1)$ where $\xi$ is a root of unity of odd order \cite{BaDaMb97}. These constructions use the notion of modified trace \cite{NgJkBp11} and relative $\mathit{G}$-modular category \cite{FcNgBp14}.
\end{abstract}
\vspace{25pt}



\section{Introduction}
	The vanishing of the dimension of an object $V$ in a ribbon category $\mathscr{C}$ is an obstruction when one studies the Reshetikhin - Turaev link invariant. If the dimension of a simple object $V$ of $\mathscr{C}$ is zero, then the quantum invariants of all (framed oriented) links with components labelled by $V$ are equal to zero, i.e. they are trivial. To overcome this difficulty, the authors N. Geer, B. Patureau-Mirand and V. Turaev introduced the notion of a modified dimension (see \cite{NgBpVt09}). The modified dimension may be non-zero when $\dim_{\mathscr{C}}(V) = 0$. Using the modified dimension, for example on the class of projective simple objects, they defined an isotopy invariant $F^{'}(L)$ (the renormalized Reshetikhin-Turaev link invariant) for any link $L$ whose components are labelled with objects of $\mathscr{C}$ under the only assumption that at least one of the labels belongs to the set of projective ambidextrous objects. Here $F^{'}(L)$ is a nontrivial link invariant (see \cite{NgBpVt09}). This modified dimension is used to construct the quantum invariants in \cite{FcNgBp14}, \cite{NgBPm12}.

	The existence of a modified dimension relates strictly with the definition of modified traces (see \cite{NgJkBp11}). In the article \cite{NgJkBp13}, the authors showed that a necessary and sufficient condition for the existence of a modified trace on an ideal generated by a simple object $J$ is that $J$ is an ambidextrous object.

	The Lie superalgebras (see \cite{Kac77}) are the generalizations of Lie algebras used by physicists to describe supersymmetry. Deformations of these superalgebras and their representations are partially known. For Lie superalgebra $\mathfrak{sl}(2|1)$, the irreducible representations of its deformations $\mathcal{U}_{\xi}\mathfrak{sl}(2|1)$ at roots of unity are described in \cite{BaDaMb97}. Using these representations and developing the idea of modified traces open up the method for constructing a quantum invariant of framed links with components labelled by irreducible representations. 
	
	The aim of this article is to construct a link invariant and a $3$-manifold invariant from quantum group $\mathcal{U}_{\xi}\mathfrak{sl}(2|1)$ at the root of unity. Note that the Lie superalgebra $\mathfrak{sl}(2|1)$ having superdimension zero, $\mathfrak{sl}(2|1)$-weight functions are trivial. Hence combining them with the Kontsevich integral or the LMO invariant also give trivial link and $3$-manifold invariants. The paper contains five sections. In section 2, we recall the monoidal category, pivotal category, braided category and, ribbon category definitions. In section 3, we describe the quantum superalgebra $\mathcal{U}_{\xi}\mathfrak{sl}(2|1)$ where ${\xi}$ is a root of unity of the odd order and by adding two elements $h_1, h_2$ to $\mathcal{U}_{\xi}\mathfrak{sl}(2|1)$, we have the Hopf superalgebra $\mathcal{U}_{\xi}^{H}\mathfrak{sl}(2|1)$. This complement helped us to construct the non semi-simple ribbon category $\mathscr{C}^{H}$ of the nilpotent simple finite dimensional representations of $\mathcal{U}_{\xi}^{H}\mathfrak{sl}(2|1)$. In section 4 we prove that a typical module over $\mathcal{U}_{\xi}^{H}\mathfrak{sl}(2|1)$ is an ambidextrous module and that a modified trace exists on the ideal of projective modules \texttt{Proj}. This modified trace will be used to construct a link invariant. In section 5, we prove that the category $\mathscr{C}^{H}$ is $\mathit{G}$-modular relative (\cite{FcNgBp14}) and we construct a $3$-manifold invariant using this property.

\subsection*{Acknowledgments}
I would like to thank B. Patureau-Mirand, my thesis advisor, who helped me with this work, and who gave me the motivation to study mathematics. I would also like to thank my professors and friends in the laboratory LMBA of the Universit\'e de Bretagne Sud.
\section{Preliminaries}       

\subsection{Monoidal category}
	\begin{Def}
	A monoidal category $\mathscr{C}$ is a category enhanced with a bifunctor called tensor product $\cdot \otimes \cdot: \mathscr{C} \times \mathscr{C} \rightarrow \mathscr{C}$ and a unity object $\mathbb{I}$ such that
	\[ \mathbb{I} \otimes \cdot = \cdot \otimes \mathbb{I} = \Id_{\mathscr{C}} \quad  \text{and} \quad  ( \cdot 		\otimes \cdot ) \otimes \cdot = \cdot \otimes ( \cdot \otimes \cdot ). \]
	\end{Def}
	We write $V \in \mathscr{C}$ to denote an object $V$ in the category $\mathscr{C}$ and call $\Hom_{\mathscr{C}}(V, W)$ the morphisms in $\mathscr{C}$ from $V \in \mathscr{C}$ to $W \in \mathscr{C}$ and $\End_{\mathscr{C}}(V)=\Hom_{\mathscr{C}}(V,V)$.

	We say that $\mathscr{C}$ is a monoidal $\mathbb{K}$-linear category if for all $V,W \in \mathscr{C}$, the morphisms $\Hom_{\mathscr{C}}(V,W)$ form a $\mathbb{K}$-module and the composition and the tensor product are bilinear and $\End_{\mathscr{C}}(\mathbb{I})\cong \mathbb{K}$. An object $V \in \mathscr{C}$ is said to be {\em simple} if $\End_{\mathscr{C}}(V)\cong \mathbb{K}$ as a unitary $\mathbb{K}$-algebra. An object $W \in \mathscr{C}$ is a direct sum of $V_{1}, ..., V_{n}\in \mathscr{C}$ if there is for $i=1, ..., n, f_{i} \in \Hom_{\mathscr{C}}(V_{i}, W), g_{i} \in \Hom_{\mathscr{C}}(W, V_{i})$ such that $g_{i}\circ f_{i}= \Id_{V_{i}}, g_{i}\circ f_{j}=0$ for $i \neq j$ and $\sum_{i=1}^{n}f_{i}\circ g_{i}=\Id_{W}$. An object $W \in \mathscr{C}$ is {\em semi-simple} if it is a direct sum of simple objects. The category $\mathscr{C}$ is {\em semi-simple} if all objects are semi-simple and $\Hom_{\mathscr{C}}(V,W)=\{0\}$ 
for any pair of non-isomorphic simple objects in $\mathscr{C}$.
\subsection{Pivotal category} 
\begin{Def}
	Let $\mathscr{C}$ be a monoidal category and $A,B \in \mathscr{C}$. A duality between $A$ and $B$ is given by a pair of morphisms $(\alpha \in \Hom_{\mathscr{C}}(\mathbb{I}, B \otimes A), \beta \in \Hom_{\mathscr{C}}(A \otimes B, \mathbb{I}))$ such that
\[ (\beta \otimes \Id_{A})\circ (\Id_{A}\otimes \alpha)=\Id_{A} \quad \text{and} \quad (\Id_{B} \otimes \beta)\circ (\alpha \otimes \Id_{B})=\Id_{B}.  \]
\end{Def}
	A {\em pivotal} category (or {\em sovereign}) is a strict monoidal category $\mathscr{C}$, with a unity object $\mathbb{I}$, equipped with the data for each object $V\in \mathscr{C}$ of its {\em dual object} $V^{*}\in \mathscr{C}$ and of four morphisms
\begin{align*}
&\ev_{V}: V^{*} \otimes V \rightarrow \mathbb{I}, &\coev_{V}: \mathbb{I} \rightarrow V \otimes V^{*},\\
&\tev_{V}:V \otimes V^{*} \rightarrow \mathbb{I}, &\tcoev_{V}:\mathbb{I} \rightarrow V^{*} \otimes V
\end{align*} 
such that $(\ev_{V}, \coev_{V})$ and $(\tev_{V}, \tcoev_{V})$ are dualities which induce the same functor duality and the same natural isomorphism $(V \otimes W)^{*}\simeq W^{*} \otimes V^{*}$. Thus, the right and left dual coincide in $\mathscr{C}$: for every morphism $h: V \rightarrow W$, we have
\begin{align*}
h^{*}&=(\ev_{W}\otimes \Id_{V^{*}}) \circ (\Id_{W^{*}} \otimes h \otimes \Id_{V^{*}}) \circ (\Id_{W^{*}} \otimes \coev_{V})\\ 
&=(\Id_{V^{*}} \otimes \tev_{W}) \circ (\Id_{V^{*}} \otimes h \otimes \Id_{W^{*}}) \circ (\tcoev_{V}\otimes \Id_{W^{*}}): W^{*} \rightarrow V^{*}
\end{align*}
and for $V,W \in \mathscr{C}$, the isomorphisms $\gamma_{V,W}: W^{*} \otimes V^{*} \rightarrow (V \otimes W)^{*}$ are given by 
\begin{align*}
\gamma_{V,W}&=(\ev_{W}\otimes \Id_{(V \otimes W)^{*}}) \circ (\Id_{W^{*}} \otimes \ev_{V} \otimes \Id_{W \otimes(V \otimes W)^{*}}) \circ (\Id_{W^{*} \otimes V^{*}} \otimes \coev_{V \otimes W})\\ 
&=(\Id_{(V \otimes W)^{*}} \otimes \tev_{V}) \circ (\Id_{(V \otimes W)^{*} \otimes V} \otimes \tev_{W} \otimes \Id_{V^{*}}) \circ (\tcoev_{V \otimes W}\otimes \Id_{W^{*}\otimes V^{*}}).
\end{align*}  

The family of isomorphisms
\[ \Phi = \{\Phi_{V}=(\tev_{V} \otimes \Id_{V^{**}}) \circ (\Id_{V} \otimes \coev_{V^{*}}): V \rightarrow V^{**}\}_{V \in \mathscr{C}} \]
is a monoidal natural isomorphism called the pivotal structure.
\subsection{Ribbon category}
	A {\em braided} category  is a tensor category $\mathscr{C}$  
provided with a braiding $c:$
for all objects $V$ and $W$ of $\mathscr{C}$, we have an isomorphism $$c_{V,W}: V \otimes W \rightarrow W \otimes V.$$
These isomorphisms are natural and for all objects $U,V$ and $W$ of $\mathscr{C}$, we have
$$c_{U,V \otimes W}=(\Id_{V}\otimes c_{U,W}) \circ (c_{U,V}\otimes \Id_{W}) \ \text{and} \ c_{U \otimes V, W}=(c_{U,W}\otimes \Id_{V})\circ (\Id_{U}\otimes c_{V,W}).$$
If the category $\mathscr{C}$ is pivotal and braided, we can define a family of natural isomorphisms  
	$$\theta_{V}=\ptr_{R}(c_{V,V})=(\Id_{V}\otimes \tev_{V})\circ (c_{V,V} \otimes \Id_{V^{*}})\circ (\Id_{V} \otimes \coev_{V}): V \rightarrow V.$$
We say that $\theta$ is a {\em twist} if it is compatible with the dual in the following sense 
	$$\forall V \in \mathscr{C}, \theta_{V^{*}}=(\theta_{V})^{*}$$
which is equivalent to $\theta_{V}=\ptr_{L}(c_{V,V})=(\ev_{V} \otimes \Id_{V} )\circ (\Id_{V^{*}} \otimes c_{V,V})\circ (\tcoev_{V} \otimes \Id_{V}): V \rightarrow V$.

A {\em ribbon category} is a braided pivotal category in which the family of isomorphisms $\theta$ is a twist.



\section{Quantum superalgebra $\mathcal{U}_{\xi}\mathfrak{sl}(2|1)$}

\subsection{Hopf superalgebra $\mathcal{U}_{\xi}\mathfrak{sl}(2|1)$}
\begin{Def}
Let $\ell\geq 3$ be an odd integer and $\xi=\exp(\frac{2\pi i}{\ell})$.
The superalgebra $\mathcal{U}_\xi\mathfrak{sl}(2|1)$ is an associative superalgebra on $\mathbb{C}$ generated by the elements $k_1,k_2,k_1^{-1},k_2^{-1}, e_1,e_2,f_1,f_2$ 
and the relations
\begin{align*}
&k_1k_2=k_2k_1,\\
&k_{i}k_{i}^{-1}=1, \ i=1,2,\\
&k_ie_jk_i^{-1}=\xi^{a_{ij}}e_j, k_if_jk_i^{-1}=\xi^{-a_{ij}}f_j \ i,j=1,2,\\
&e_1f_1-f_1e_1=\frac{k_1-k_1^{-1}}{\xi-\xi^{-1}}, e_2f_2+f_2e_2=\frac{k_2-k_2^{-1}}{\xi-\xi^{-1}},\\
&[e_1, f_2]=0, [e_2, f_1]=0,\\
&e_2^{2}=f_2^{2}=0,\\
&e_1^{2}e_2-(\xi+\xi^{-1})e_1e_2e_1+e_2e_1^{2}=0,\\
&f_1^{2}f_2-(\xi+\xi^{-1})f_1f_2f_1+f_2f_1^{2}=0.
\end{align*}
The last two relations are called the Serre relations. The matrix $(a_{ij})$ is given by $a_{11}=2, a_{12}=a_{21}=-1, a_{22}=0$. The odd generators are $e_2, f_2$. 
\end{Def}
 We define $\xi^{x}:=\exp(\frac{2\pi i x}{\ell})$, afterwards we will use the concepts
 $$\{x\}= \xi^{x}-\xi^{-x}, [x]= \dfrac{\xi^{x}-\xi^{-x}}{\xi-\xi^{-1}}.$$

Set $e_3=e_1e_2-\xi^{-1}e_2e_1, f_3=f_2f_1-\xi f_1f_2$. The Serre relations become
$$e_1e_3 = \xi e_3e_1, f_3f_1=\xi^{-1}f_1f_3.$$ 
Furthermore 
\begin{align*}
&e_2e_3=-\xi e_3e_2, \ f_3f_2=-\xi^{-1}f_2f_3, \\
&e_3f_3+f_3e_3=\frac{k_1k_2-k_1^{-1}k_2^{-1}}{q-q^{-1}},\\
&e_3^{2}=f_3^{2}=0.
\end{align*}
	According to \cite{SMkVNt91}, $\mathcal{U}_\xi\mathfrak{sl}(2|1)$ is a Hopf superalgebra with the coproduct, counit and antipode as below
\begin{align*}
&\Delta(e_i)=e_i \otimes 1 + k_i^{-1} \otimes e_i \ i=1,2,\\ 
&\Delta(f_i)=f_i \otimes k_i + 1 \otimes f_i \ i=1,2,\\
&\Delta(k_i)=k_i \otimes k_i \ i=1,2,\\
&S(e_i)=- k_ie_i, S(f_i)=-f_ik_i^{-1}, S(k_i)=k_i^{-1} \ i=1,2,\\
&\epsilon(k_i)=1, \epsilon(e_i)=\epsilon(f_i)=0 \ i=1,2.
\end{align*}
The center and representations of $\mathcal{U}_\xi\mathfrak{sl}(2|1)$ were studied by B. Abdesselam, D. Arnaudon and M. Bauer \cite{BaDaMb97}. We focus on the case of nilpotent representations of type $\mathfrak{B}$ with the condition $\ell$ odd.
\begin{rmq}
\begin{enumerate}
\item Because $(e_1 \otimes 1)(k_1^{-1} \otimes e_1)=\xi^2(k_1^{-1} \otimes e_1)(e_1 \otimes 1)$ and $\left(\ell \right)_{\xi}:=\frac{1-\xi^{\ell}}{1-\xi}=0$ then $\Delta(e_1^{\ell})=\sum_{m=0}^{\ell}\binom{\ell}{m}_{\xi}(e_1 \otimes 1)^{m}(k_1^{-1} \otimes e_1)^{\ell - m} = e_1^\ell \otimes 1 + k_1^{-\ell} \otimes e_1^\ell$. We have $\Delta^{op}(e_1^{\ell})=1 \otimes e_1^\ell + e_1^\ell \otimes k_1^{-\ell}$ at the same time. It is known that $e_1^\ell, f_1^\ell, k_1^\ell \in \mathcal{Z}$ where $\mathcal{Z}$ is the center of $\mathcal{U}_\xi\mathfrak{sl}(2|1)$, so $\Delta(e_1^{\ell}) \in \mathcal{Z} \otimes \mathcal{Z}$. It follows that there exists no element $R \in \mathcal{U}_\xi\mathfrak{sl}(2|1) \otimes \mathcal{U}_\xi\mathfrak{sl}(2|1)$ such that $\Delta^{op}(x)=R\Delta(x)R^{-1} \ \forall x\in \mathcal{U}_\xi\mathfrak{sl}(2|1)$, i.e. the superalgebra $\mathcal{U}_\xi\mathfrak{sl}(2|1)$ is not quasitriangular.
\item We think that the quotient superalgebra $\mathcal{U}_\xi\mathfrak{sl}(2|1)/(e_1^\ell,f_1^\ell)$ is not quasitriangular but $\mathcal{U}_\xi\mathfrak{sl}(2|1)/(e_1^\ell,f_1^\ell,k_1^\ell-1, k_2^\ell-1)$ should be, a proof of this might be found along the lines of \cite{lentner2014new} where the author uses a version of quantum group with divided power. This is not the quotient that interests us in this article.
\item The unrolled version $\mathcal{U}_{\xi}^{H}\mathfrak{sl}(2|1)$ seems to be quasitriangular only in a topological sense (see \cite{Ha_en_cours}). However, we will show in Theorem \ref{ptur} and Proposition \ref{twist} that some representations (the weight modules) form a ribbon category.
\end{enumerate}
\end{rmq}

The superalgebra $\mathcal{U}_\xi\mathfrak{sl}(2|1)/(e_1^\ell,f_1^\ell)$ has a Poincar\'e-Birkhoff-Witt basis $\{e_2^{\rho}e_3^{\sigma}e_1^{p}k_1^{s}k_2^{t}f_2^{\rho'}f_3^{\sigma'}f_1^{p'}, \rho, \sigma,\rho', \sigma' \in \{0,1\}, p, p' \in \{0,1, ...,\ell-1\}, s, t \in \mathbb{Z}\}$, its Borel part is a superalgebra $\mathcal{U}_\xi(\mathfrak{n}_{+})$ which has a vector space basis $\{e_2^{\rho}e_3^{\sigma}e_1^{p}\ \rho, \sigma \in \{0,1\}, p \in \{0,1, ...,\ell-1\}\}$. It is well known that $\mathcal{U}_\xi(\mathfrak{n}_{+})$ is a Nichols algebra of diagonal type associated with the generalized Dynkin diagram $\epsh{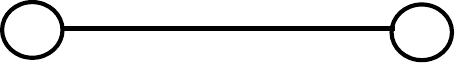}{1ex}
      \put(-7,2){\ms{-1}} \put(-39,4){\ms{\xi^2}} \put(-22,3){\ms{\xi^{-2}}}\ \ $(see \cite{Hec2008}). We now explain this point of view. We consider the group algebra $B=\mathbb{C}G$ in which $G$ is an abelian group generated by $k_1, k_2$, a vector space $V$ on $\mathbb{C}$ generated by $e_1, e_2$. Here $B$ is a Hopf algebra and $(V, \cdot, \delta)$ is a Yetter-Drinfeld module on $B$ \cite{Hec2008}, where the action $\cdot: B \otimes V \rightarrow V$ of $B$ on $V$ is determined by 
\begin{align*}
&k_1 \cdot e_1 = \xi^{2}e_1,\ k_1 \cdot e_2 = \xi^{-1}e_2,\\  
&k_2 \cdot e_1 = \xi^{-1}e_1, \ k_2 \cdot e_2 = -e_2, 
\end{align*}
the matrix determining the bicharacter is $(q_{ij})_{2\times 2}, q_{ij}=(-1)^{\vert i\vert \vert j\vert}\xi^{a_{ij}}$ where $\vert 1\vert=0$, $\vert 2\vert=1$ and the coaction $\delta: V \rightarrow B \otimes V$ of $B$ on $V$ is given by
	$$\delta(e_i)=k_i \otimes e_i\ i=1,2.$$
It is clear that $\delta(b \cdot v)=b_{(1)}v_{(-1)}S(b_{(3)}) \otimes b_{(2)}\cdot v_{(0)}$ for all $b \in B, v \in V$. Here we use the Sweedler notation and write $(\Delta \otimes \Id)\Delta(b)=b_{(1)} \otimes b_{(2)} \otimes b_{(3)}, \ \delta(v)=v_{(-1)}\otimes v_{(0)}$ for $b \in B, v \in V$.\\
Using Hopf algebra $B$ and Yetter-Drinfeld module $V$ we can determine the Nichols algebra $\mathcal{B}(V)=T(V)/\mathcal{J}(V)$ where $T(V)=\bigoplus_{n=0}^{\infty}V^{\otimes n}$ is the tensor algebra of $V$ with the braided copoduct $\widetilde{\Delta}(v)=1\otimes v + v \otimes 1$ and counit $\epsilon(v)=0$ for $v \in V$, $\mathcal{J}(V)$ is the maximal coideal of $T(V)$. We now check that $e_2^{2}$ and the Serre relation $w=e_1e_3-\xi e_3e_1$ are in $\mathcal{J}(V)$. We have $\widetilde{\Delta}(e_2^{2})=\widetilde{\Delta}(e_2)\widetilde{\Delta}(e_2)=(1 \otimes e_2 + e_2 \otimes 1)(1 \otimes e_2 + e_2 \otimes 1)=1 \otimes e_2^{2}+(k_2\cdot e_2) \otimes e_2+e_2\otimes e_2+e_2^{2} \otimes 1=1 \otimes e_2^{2} + e_2^{2} \otimes 1$, so $e_2^{2} \in \mathcal{J}(V)$. 

We calculate
\begin{align*}
&\widetilde{\Delta}(e_3)=\widetilde{\Delta}(e_1)\widetilde{\Delta}(e_2)-\xi^{-1}\widetilde{\Delta}(e_2)\widetilde{\Delta}(e_1)\\
&=(1 \otimes e_1 + e_1 \otimes 1)(1 \otimes e_2 + e_2 \otimes 1)-\xi^{-1}(1 \otimes e_2 + e_2 \otimes 1)(1 \otimes e_1 + e_1 \otimes 1)\\
&=1\otimes e_1e_2 + (k_1\cdot e_2) \otimes e_1+ e_1 \otimes e_2 + e_1e_2 \otimes 1\\
&\qquad \qquad -\xi^{-1}\left(1\otimes e_2e_1 + (k_2\cdot e_1) \otimes e_2 + e_2 \otimes e_1 + e_2e_1 \otimes 1 \right)\\
&=1 \otimes e_3 + e_3 \otimes 1 + (1-\xi^{-2})e_1\otimes e_2.
\end{align*}
And a similar calculation gives us
\begin{align*}
\widetilde{\Delta}(e_1)\widetilde{\Delta}(e_3)&=1 \otimes e_1e_3 + \xi e_3\otimes e_1+(1-\xi^{-2})\xi^{2}e_1\otimes e_1e_2 \\
&\qquad + e_1\otimes e_3+e_1e_3 \otimes 1  +(1-\xi^{-2})e_1^{2}\otimes e_2,
\end{align*}
and
\begin{align*}
\widetilde{\Delta}(e_3)\widetilde{\Delta}(e_1)&=1 \otimes e_3e_1 + \xi e_1\otimes e_3 + e_3\otimes e_1+e_3e_1 \otimes 1 \\
&\qquad +(1-\xi^{-2})e_1\otimes e_2e_1 +(1-\xi^{-2})\xi^{-1}e_1^{2}\otimes e_2.
\end{align*}
Thus we have 
\begin{align*}
\widetilde{\Delta}(w)&=\widetilde{\Delta}(e_1)\widetilde{\Delta}(e_3)-\xi \widetilde{\Delta}(e_3)\widetilde{\Delta}(e_1)\\
&=1 \otimes w+w\otimes 1+ (\xi^{2}-1)e_1\otimes e_1e_2+e_1 \otimes e_3\\
&\qquad -\xi^{2}e_1 \otimes e_3-(\xi-\xi^{-1})e_1\otimes e_2e_1\\
&=1 \otimes w+w\otimes 1.
\end{align*}
This implies that $w\in \mathcal{J}(V)$. The bosonization of $\mathcal{B}(V)$ is then isomorphic to a Hopf subalgebra of the bosonization of the Hopf superalgebra $\mathcal{U}_{\xi}\mathfrak{sl}(2|1)$.  

\subsection{Pivotal Hopf superalgebra $\mathcal{U}_\xi\mathfrak{sl}(2|1)$} \label{2.2}
\begin{Pro}
Given $\phi_0=k_1^{-\ell}k_2^{-2}$, so $\forall u \in \mathcal{U}_\xi\mathfrak{sl}(2|1), S^{2}(u)=\phi_0 u \phi_0^{-1}$.
\end{Pro}
\begin{proof}
This can be verified for generator elements $k_i, e_i, f_i$,  $i=1,2$.
\end{proof}
It follows that the Hopf superalgebra $\mathcal{U}_\xi\mathfrak{sl}(2|1)$ provided with the pivotal element $\phi_0=k_1^{-\ell}k_2^{-2}$ is pivotal superalgebra (see \cite{BePM12}). 

	Given $\mathscr{C}$ the even category of representations of $\mathcal{U}_\xi\mathfrak{sl}(2|1)$ in $\mathbb{C}$-vector spaces of finite dimension, the category $\mathscr{C}$ is pivotal. If $V$ is an object of $\mathscr{C}$, its dual is a $\mathbb{C}$-vector space $V^{*}=\Hom_{\mathbb{C}}(V, \mathbb{C})$ provided with the action of $u$ given by $(u,\varphi)\mapsto (-1)^{\deg u\deg\varphi}\varphi \circ \rho_{V}(S(u))$ where $\rho_{V}:\mathcal{U}_\xi\mathfrak{sl}(2|1) \rightarrow \End_{\mathbb{C}}(V)$ is the representation of $\mathcal{U}_\xi\mathfrak{sl}(2|1)$.

	The unity element of category $\mathscr{C}$ is the module $\mathbb{C}$ provided with the representation $\epsilon: \mathcal{U}_\xi\mathfrak{sl}(2|1) \rightarrow \mathbb{C}\cong \End_{\mathbb{C}}(\mathbb{C})$. 
If one has a basis $(e_{i})_{i}$ of $V$ with dual basis $(e_{i}^{*})_{i}$, it can be described by dual morphisms
\begin{align*}
\ev_{V}: &e_{i}^{*} \otimes e_{j} \mapsto e_{i}^{*}(e_{j})=\delta_{i}^{j}, & \coev_{V}: &1 \mapsto \sum_{i}{e_{i} \otimes e_{i}^{*}}, \\
\tev_{V}: &e_{j} \otimes  e_{i}^{*} \mapsto (-1)^{\deg e_{j}}e_{i}^{*}(\phi_{0}.e_{j}),&
\tcoev_{V}: &1 \mapsto \sum_{i}{e_{i}^{*} \otimes (-1)^{\deg e_{i}}(\phi_{0}^{-1}.e_{i})}.
\end{align*}
\subsection{Category of nilpotent weight modules}
\subsubsection{Typical module}
	We consider the even category $\mathscr{C}$ of the nilpotent finite dimensional representations over $\mathcal{U}_{\xi}\mathfrak{sl}(2|1)$, its objects are finite dimensional representations of $\mathcal{U}_\xi\mathfrak{sl}(2|1)$ on which $e_{1}^{\ell}=f_{1}^{\ell}=0$ and $k_1, k_2$ are diagonalizable. If $V, V^{'} \in \mathscr{C}$, $\Hom_{\mathscr{C}}(V, V^{'})$ is formed by the even morphisms between these two modules (see \cite{NgBp08}). Each nilpotent simple module (called "of type $\mathfrak{B}$" in section 5.2 \cite{BaDaMb97}) is determined by the highest weight $\mu = (\mu_{1}, \mu_{2}) \in \mathbb{C}^{2}$ and is denoted $ V_{\mu_{1}, \mu_{2}}$ or $V_{\mu}$. Its highest weight vector $w_{0, 0, 0}$ satisfies
\begin{align*} 
e_1w_{0,0,0}&=0, &e_2w_{0,0,0}&=0, \\
k_1w_{0,0,0}&=\lambda_1 w_{0,0,0}, &k_2w_{0,0,0}&=\lambda_2 w_{0,0,0}
\end{align*}
where $\lambda_i=\xi^{\mu_i}$ with $i=1,2.$
	
	For $\mu=(\mu_{1}, \mu_{2}) \in \mathbb{C}^{2}$ we say that $\mathcal{U}_\xi\mathfrak{sl}(2|1)$-module $V_{\mu}$ is typical if it is a simple module of dimension $4\ell$. Other simple modules are said to be atypical.
	
	The basis of a typical module is formed by vectors $w_{\rho, \sigma, p}=f_2^{\rho}f_3^{\sigma}f_1^{p}w_{0,0,0}$ where $\rho, \sigma \in \{0, 1\}, 0\leq p < \ell$. The odd elements are $w_{0,1, p}$ and $w_{1,0, p}$, others are even. The representation of typical $\mathcal{U}_\xi\mathfrak{sl}(2|1)$-module $V_{\mu_1, \mu_2}$ is determined by
\begin{align*}
&k_{1}w_{\rho, \sigma, p}=\lambda_{1}\xi^{\rho-\sigma-2p}w_{\rho, \sigma, p}, \\
&k_{2}w_{\rho, \sigma, p}=\lambda_{2}\xi^{\sigma+p}w_{\rho, \sigma, p},\\
&f_{1}w_{\rho, \sigma, p}=\xi^{\sigma-p}w_{\rho, \sigma, p+1}-\rho(1-\sigma)\xi^{-\sigma}w_{\rho-1, \sigma +1, p},\\
&f_{2}w_{\rho, \sigma, p}=(1-\rho)w_{\rho+1, \sigma, p}, \\
&e_{1}w_{\rho, \sigma, p}=-\sigma(1-\rho)\lambda_{1}\xi^{-2p+1}w_{\rho+1, \sigma-1, p}+[p][\mu_{1}-p+1]w_{\rho, \sigma, p-1},\\
&e_{2}w_{\rho, \sigma, p}=\rho[\mu_{2}+p+\sigma]w_{\rho -1, \sigma, p}+\sigma(-1)^{\rho}\lambda_{2}^{-1}\xi^{-p}w_{\rho, \sigma -1, p+1}.
\end{align*}
where $\rho, \sigma \in \{0,1 \}$ and $p \in \{0,1,...,\ell-1\}$. \vspace{13pt}

	We also have $V_{\mu}\simeq V_{\mu+\vartheta}\Leftrightarrow \vartheta \in (\ell \mathbb{Z})^{2}$.

\begin{rmq}	 
	The module $V_{\mu}$ is typical if $[\mu_{1}-p+1] \ne 0 \ \forall p \in \{1,...,\ell-1\} \ (\mu_{1} \neq p-1+\frac{\ell}{2}\mathbb{Z}\ \forall p \in \{1,...,\ell-1\})$ and $[\mu_{2}][\mu_{1}+\mu_{2}+1] \ne 0 \ (\mu_2 \neq \frac{\ell}{2}\mathbb{Z}, \mu_{1}+\mu_{2} \neq -1 + \frac{\ell}{2}\mathbb{Z})$ {\em (see \cite{BaDaMb97})}.
\end{rmq}
	We define a $\mathbb{C}$-superalgebra $\mathcal{U}_{\xi}^{H}\mathfrak{sl}(2|1)$ (note $\mathcal{U}^{H}$) as $\mathcal{U}_{\xi}^{H}\mathfrak{sl}(2|1)=\left< \mathcal{U}_\xi\mathfrak{sl}(2|1), h_{i}, i=1,2\right>$ with the relations in $\mathcal{U}_\xi\mathfrak{sl}(2|1)$ and $[h_{i},e_{j}]=a_{ij}e_{j}, [h_{i},f_{j}]=-a_{ij}f_{j}$, $[h_{i},h_{j}]=0, [h_{i},k_{j}]=0 \ i, j = 1, 2$.
	
	The superalgebra $\mathcal{U}^{H}$ is a Hopf superalgebra where $\Delta, S$ and $\epsilon$ are determined as in $\mathcal{U}_\xi\mathfrak{sl}(2|1)$ and by 
	$$\Delta(h_i)=h_i \otimes 1 + 1 \otimes h_i, S(h_i)=-h_i, \epsilon(h_i)=0 \ i=1,2.$$ 
	We consider the even category $\mathscr{C}^{H}$ of nilpotent finite dimensional $\mathcal{U}^{H}$-modules (precise that $e_{1}^{\ell}=f_{1}^{\ell}=0$) for which $\xi^{h_i}=k_i$ as diagonalizable operators. The category $\mathscr{C}^{H}$ is pivotal similar to $\mathscr{C}$ (see Section \ref{2.2}).
	
	We define the actions of $h_i, i=1, 2$ on the basis of $V_{\mu_{1},\mu_{2}}$ by $$h_1w_{\rho, \sigma, p}=(\mu_1+\rho - \sigma - 2p)w_{\rho, \sigma, p}, h_2w_{\rho, \sigma, p}=(\mu_2+\sigma + p)w_{\rho, \sigma, p}.$$
	Thus $V_{\mu_{1},\mu_{2}}$ is a weight module of $\mathscr{C}^{H}.$ A module in $\mathscr{C}^{H}$ is said to be typical if, seen as a $\mathcal{U}_\xi\mathfrak{sl}(2|1)$-module, it is typical. For each module $V$ we note $\overline{V}$ the same module with the opposite parity.
	We set $\mathit{G}=\mathbb{C} / \mathbb{Z} \times \mathbb{C}/ \mathbb{Z}$ and for each $\overline{\mu} \in \mathit{G}$ we define $\mathscr{C}_{\overline{\mu}}^{H}$ as the sub-category of weight modules which have their weights in the coset $\overline{\mu} \ (\text{modulo} \ \mathbb{Z}\times \mathbb{Z})$. So $\{ \mathscr{C}_{\overline{\mu}}^{H} \}_{\overline{\mu} \in \mathit{G}}$ is a $\mathit{G}$-graduation (where $\mathit{G}$ is an additive group): let $V \in \mathscr{C}_{\overline{\mu}}^{H}, V^{'}\in \mathscr{C}_{\overline{\mu}^{'}}^{H}$, then the weights of $V\otimes V^{'}$ are congruent to $\overline{\mu} + \overline{\mu}^{'}\ (\text{modulo} \ \mathbb{Z}\times \mathbb{Z})$. Furthermore, if $\overline{\mu} \ne \overline{\mu}^{'}$ then $\Hom_{\mathscr{C}^{H}}(V,V^{'})=0$ because a morphism preserves weights. 
	
	We also define $$\mathit{G}_{s}=\{\overline{g} \in \mathit{G} \ \textit{such that} \  \exists V \in \mathscr{C}_{\overline{g}}^{H} \ \textit{simple of} \ \mathscr{C}^{H} \textit{and atypical}\}.$$
	It follows from \cite{BaDaMb97} that $$\mathit{G}_{s}=\left\{\overline{0}, \overline{\frac{1}{2}}\right\}\times \mathbb{C}/ \mathbb{Z} \cup \mathbb{C}/ \mathbb{Z} \times \left \{\overline{0}, \overline{\frac{1}{2}}\right\} \cup \left\{(\overline{\mu_1}, \overline{\mu_2}): \overline{\mu_1} + \overline{\mu_2}\in\left\{\overline{0}, \overline{\frac{1}{2}}\right\}\right\}.$$
\subsubsection{Character of representations of $\mathcal{U}_\xi^{H}\mathfrak{sl}(2|1)$}
\begin{Def}
	The character of a weight module $V$ is $$\chi_{V}=\sum_{\mu}\dim(E_{\mu}(V))X_{1}^{\mu_{1}}X_{2}^{\mu_{2}}$$
where $E_{\mu}(V)$ is the proper subspace of the proper value $\mu=(\mu_{1}, \mu_{2})$ of $(h_1, h_2)$.
\end{Def}
Note that we do not use the concept of a super-character defined as above by replacing the dimension by the super-dimension. 

	A finite dimensional representation of $\mathcal{U}_\xi\mathfrak{gl}(2)$, subalgebra generated by $e_{1}, f_{1}, k_{i}$ is defined by $V=Vect(v_{0}, ..., v_{\ell-1})$ \cite{BaDaMb97} 
\begin{align*}
&k_1v_p=\lambda_1 \xi^{-2p}v_p \ \text{with} \ p \in \{0,1, ..., \ell-1\},\\
&f_1v_p = v_{p+1} \ \text{with} \ p \in \{0,1, ..., \ell-2\} \ \text{and} \ f_1v_{\ell-1}=0,\\
&e_1v_p=[p][\mu_1-p+1]v_{p-1},  \ \xi^{\mu_1} \equiv \lambda_1,\\
&k_2v_p=\lambda_2 \xi^{p}v_p \ \text{with} \ p \in \{0,1, ..., \ell-1\}.
\end{align*}
	It extends to the generators $h_{1}, h_{2}$ by
\begin{align*}
&h_1v_p=(\mu_1 -2p)v_p \ \text{with} \ p \in \{0,1, ..., \ell-1\}\\
&h_2v_p=(\mu_2 + p)v_p \ \text{with} \ p \in \{0,1, ..., \ell-1\}
\end{align*}
so that $\xi^{h_i}=k_i, i=1,2$ on $V$.
	We have the character of representation of $\mathcal{U}_\xi\mathfrak{gl}(2)$
\begin{equation*}
\chi_{V_{\mu_1, \mu_2}^{\mathfrak{gl}(2)}}=X_1^{\mu_1}X_2^{\mu_2}\frac{1-x^{\ell}}{1-x} \ \text{where} \ x=X_1^{-2}X_2.
\label{<>}
\end{equation*}

	In the case of a typical representation, the nilpotent representation $V_{\mu_{1},\mu_{2}}$ of $\mathcal{U}_\xi\mathfrak{sl}(2|1)$ with highest weight $(\mu_{1}, \mu_{2})$ is determined by
\begin{align*}
&k_1w_{\rho, \sigma, p}=\lambda_1\xi^{\rho - \sigma - 2p}w_{\rho, \sigma, p},\\ 
&k_2w_{\rho, \sigma, p}=\lambda_2\xi^{\sigma + p}w_{\rho, \sigma, p}
\end{align*}
with $h_1w_{\rho, \sigma, p}=(\mu_1+\rho - \sigma - 2p)w_{\rho, \sigma, p}$ and $h_2w_{\rho, \sigma, p}=(\mu_2+ \sigma + p)w_{\rho, \sigma, p}$.
So the nilpotent representation $V_{\mu_{1},\mu_{2}}$ has the following character
\begin{multline} \label{dac so sl21}
\chi_{V_{\mu_1, \mu_2}^{\mathfrak{sl}(2|1)}}=\chi_{V_{\mu_1, \mu_2, \rho = \sigma = 0}^{\mathfrak{gl}(2)}}+\chi_{V_{\mu_1, \mu_2, \rho = 1,\sigma = 0}^{\mathfrak{gl}(2)}}+\chi_{V_{\mu_1, \mu_2, \rho = 0,\sigma = 1}^{\mathfrak{gl}(2)}}+\chi_{V_{\mu_1, \mu_2, \rho = \sigma = 1}^{\mathfrak{gl}(2)}}\\
= X_1^{\mu_1}X_2^{\mu_2}\frac{1-x^{\ell}}{1-x}(1+X_1)(1+X_{1}x).
\end{multline}
\subsubsection{Braided category $\mathscr{C}^{H}$}
	The $\mathbb{C}$-superalgebra $\mathcal{U}_\xi\mathfrak{sl}(2|1)$	can be seen as the specialisation at $q=\xi$ of the $\mathbb{C}(q)$-subsuperalgebra $\mathcal{U}_{q}\mathfrak{sl}(2|1)$ of the $h$-adic quantized enveloping superalgebra of $\mathfrak{sl}(2|1)$ where $q=e^{h} \in \mathbb{C}[[h]]$.	In articles \cite{SMkVNt91, HYa94} the authors showed that: 
$\mathcal{R}^{q}=\check{\mathcal{R}}^{q}\mathcal{K}_{q}$ where $$\check{\mathcal{R}}^{q}=\sum_{i=0}^{\infty}\frac{\{1\}^{i}e_1^{i} \otimes f_1^{i}}{(i)_{q}!} \sum_{j=0}^{1}\frac{(-\{1\})^{j}e_3^{j} \otimes f_3^{j}}{(j)_{q}!} \sum_{k=0}^{1}
\frac{(-\{1\})^{k}e_2^{k} \otimes f_2^{k}}{(k)_{q}!},$$ $(0)_{q}!=1, (n)_{q}!:=(1)_{q}(2)_{q}\ldots(n)_{q}, (k)_{q}=\frac{1-q^k}{1-q}$ and $\mathcal{K}_{q}=q^{-h_1 \otimes h_2 -h_2 \otimes h_1 - 2h_2 \otimes h_2}$ is a universal $R$-matrix element of superalgebra $\mathcal{U}_{q}\mathfrak{sl}(2|1)$. That is, we have the following relations
$$(\Delta \otimes \Id)(\mathcal{R}^{q})=\mathcal{R}_{13}^{q}\mathcal{R}_{23}^{q}, \
(\Id \otimes \Delta)(\mathcal{R}^{q})=\mathcal{R}_{13}^{q}\mathcal{R}_{12}^{q}, \
\Delta^{op}(x)\mathcal{R}^{q}=\mathcal{R}^{q}\Delta(x)$$ 
for all $x \in \mathcal{U}_q\mathfrak{sl}(2|1)$. The superalgebra $\mathcal{U}_q\mathfrak{sl}(2|1)$ has a Poincar\'e-Birkhoff-Witt basis $\{e_{1}^{p^{'}}e_{3}^{\sigma^{'}}e_{2}^{\rho^{'}}h_{1}^{s_{1}}h_{2}^{s_{2}}f_{2}^{\rho}f_{3}^{\sigma}f_{1}^{p}, p, p^{'} \in \mathbb{N}, \rho, \sigma, \rho^{'},\sigma^{'} \in \{ 0,1 \}, \\s_{1}, s_{2} \in \mathbb{N}  \}$.
Using this basis we can write $\mathcal{U}_q\mathfrak{sl}(2|1)$ as a direct sum $\mathcal{U}_q\mathfrak{sl}(2|1)=\mathcal{U}^{<} \oplus I$ where $\mathcal{U}^{<}$ is a $\mathbb{C}(q)$-module generated by the elements $e_{1}^{p^{'}}e_{3}^{\sigma^{'}}e_{2}^{\rho^{'}}h_{1}^{s_{1}}h_{2}^{s_{2}}f_{3}^{\sigma}f_{2}^{\rho}f_{1}^{p}$ for $0\leq p, p^{'} < \ell; \rho, \sigma, \rho^{'},\sigma^{'} \in \{ 0,1 \}, s_{1}, s_{2} \in \mathbb{N} $ and $I$ is generated by the other monomials.
Set $p: \mathcal{U}_q\mathfrak{sl}(2|1) \rightarrow \mathcal{U}^{<}$ the projection with kernel $I$. We define $$\mathcal{R}^{<}= p\otimes p(\mathcal{R}^{q})=p \otimes \Id(\mathcal{R}^{q})=\Id \otimes p(\mathcal{R}^{q}).$$
The proposition below shows that the "truncated R-matrix" $\mathcal{R}^{<}$ satisfies the properties of an R-matrix "modulo truncation".
\begin{Pro} \label{mdR}
$\mathcal{R}^{<}$ satisfies:
\begin{enumerate}
\item $(p \otimes p \otimes p)(\Delta \otimes \Id(\mathcal{R}^{<}))=(p \otimes p \otimes p)\mathcal{R}_{13}^{<}\mathcal{R}_{23}^{<}$,
\item $(p \otimes p \otimes p)(\Id \otimes \Delta(\mathcal{R}^{<}))=(p \otimes p \otimes p)\mathcal{R}_{13}^{<}\mathcal{R}_{12}^{<}$,
\item $(p \otimes p)(\mathcal{R}^{<} \Delta^{op}(x))=(p \otimes p)(\Delta(x)\mathcal{R}^{<})$ for all $x \in \mathcal{U}_q\mathfrak{sl}(2|1)$.
\end{enumerate}
\end{Pro}
\begin{proof}
The above relations and $p \circ p=p$ give us $(p \otimes p \otimes p)(\Delta \otimes \Id(\mathcal{R}^{q}))=(p \otimes p \otimes p)(\Delta \otimes \Id)(\Id \otimes p(\mathcal{R}^{q}))=(p \otimes p \otimes p)(\Delta \otimes \Id)(\mathcal{R}^{<})$. At the same time $(p \otimes p \otimes p)(\mathcal{R}_{13}^{q}\mathcal{R}_{23}^{q})=(p \otimes p \otimes p)((p \otimes \Id \otimes \Id)(\mathcal{R}_{13}^{q})(\Id \otimes p \otimes \Id)(\mathcal{R}_{23}^{q}))=(p \otimes p \otimes p)(\mathcal{R}_{13}^{<}\mathcal{R}_{23}^{<})$. 
So
\begin{equation} 
(p \otimes p \otimes p)(\Delta \otimes \Id(\mathcal{R}^{<}))=(p \otimes p \otimes p)\mathcal{R}_{13}^{<}\mathcal{R}_{23}^{<}.
\end{equation}
Similarly we also have
\begin{equation} 
(p \otimes p \otimes p)(\Id \otimes \Delta)(\mathcal{R}^{<})=(p \otimes p \otimes p)(\mathcal{R}_{13}^{<}\mathcal{R}_{12}^{<}).
\end{equation}
For the third equality, it is enough to check on the generator elements.

It is true when $x=h_{i}$ 
because $\Delta(h_{i})$ is symmetric and $\Delta(h_{i})(e_{j} \otimes f_{j})=e_{j} \otimes h_{i}f_{j}+h_{i}e_{j} \otimes f_{j}=e_{j} \otimes f_{j}(h_{i}-a_{ij})+e_{j}(h_{i}+a_{ij}) \otimes f_{j}=e_{j} \otimes f_{j}(1 \otimes (h_{i}-a_{ij})+(h_{i}+a_{ij}) \otimes 1)=(e_{j} \otimes f_{j})\Delta(h_{i})$.

For $x=e_{i}$ we have $(p \otimes p)(\Delta^{op}(e_{i})\mathcal{R}^{q})=(p \otimes p)(1 \otimes e_{i} + e_{i} \otimes k_{i}^{-1})\mathcal{R}^{q}=(p \otimes p)((1 \otimes e_{i})\mathcal{R}^{q})+(p \otimes p)((e_{i} \otimes k_{i}^{-1})\mathcal{R}^{q})=(p \otimes p)((1 \otimes e_{i})\mathcal{R}^{<})+(p \otimes p)((e_{i} \otimes k_{i}^{-1})\mathcal{R}^{<})=(p \otimes p)(\Delta^{op}(e_{i})\mathcal{R}^{<})$. On the other side $(p \otimes p)(\mathcal{R}^{q}\Delta(e_{i}))=(p \otimes p)(\mathcal{R}^{<}\Delta(e_{i}))$. So we have $(p \otimes p)(\Delta^{op}(e_{i})\mathcal{R}^{<})=(p \otimes p)(\mathcal{R}^{<}\Delta(e_{i}))$.

 For $x=f_{i}$ we proceed analogously.
So we deduce that $$(p \otimes p)(\Delta^{op}(x)\mathcal{R}^{<})=(p \otimes p)(\mathcal{R}^{<}\Delta(x)) \ \forall x \in \mathcal{U}_{q}\mathfrak{sl}(2|1).$$
\end{proof}

Let $\mathcal{K}$ be the operator in $\mathscr{C}^{H} \otimes \mathscr{C}^{H}$ defined by
$$\mathcal{K}=\xi^{-h_1 \otimes h_2 -h_2 \otimes h_1 - 2h_2 \otimes h_2}$$
that is $\forall V, W \in \mathscr{C}^{H}, \mathcal{K}_{V \otimes W}=\exp{(\rho_{V \otimes W}(\frac{2i\pi}{\ell}(-h_1 \otimes h_2 -h_2 \otimes h_1 - 2h_2 \otimes h_2)))}$ is a linear map on the finite dimensional vector space $V \otimes W$. For example, if $w_{\rho, \sigma, p} \otimes w_{\rho^{'}, \sigma^{'}, p^{'}} \in V_\mu \otimes V_{\mu^{'}}$, one has 

$\mathcal{K}_{V \otimes W}(w_{\rho, \sigma, p} \otimes w_{\rho^{'}, \sigma^{'}, p^{'}})\\
=\xi^{-(\mu_{1}+\rho-\sigma-2p)(\mu_{2}^{'}+\sigma^{'}+p^{'})-(\mu_{2}+\sigma+p)(\mu_{1}^{'}+\rho^{'}-\sigma^{'}-2p^{'})-2(\mu_{2}+\sigma+p)(\mu_{2}^{'}+\sigma^{'}+p^{'})}w_{\rho, \sigma, p} \otimes w_{\rho^{'}, \sigma^{'}, p^{'}}$.
We have
\begin{equation} \label{ptlh}
\Delta \otimes \Id(\mathcal{K})=\mathcal{K}_{13}\mathcal{K}_{23}, \Id \otimes \Delta(\mathcal{K})=\mathcal{K}_{13}\mathcal{K}_{12}.
\end{equation}

Let $\check{\mathcal{R}}^{<}$ be the universal truncated quasi $R$-matrix of $\mathcal{U}_{q}\mathfrak{sl}(2|1), q=e^{h} \in \mathbb{C}[[h]]$ given by $\check{\mathcal{R}}^{<}=p \otimes p(\check{\mathcal{R}}^{q})=\Id \otimes p(\check{\mathcal{R}}^{q})=p \otimes \Id(\check{\mathcal{R}}^{q})$, i.e:
$$\check{\mathcal{R}}^{<}=\sum_{i=0}^{\ell-1}\frac{\{1\}^{i}e_1^{i} \otimes f_1^{i}}{(i)_{q}!} \sum_{j=0}^{1}\frac{(-\{1\})^{j}e_3^{j} \otimes f_3^{j}}{(j)_{q}!} \sum_{k=0}^{1}\frac{(-\{1\})^{k}e_2^{k} \otimes f_2^{k}}{(k)_{q}!}.$$
Set $\check{\mathcal{R}}=\check{\mathcal{R}}^{<}\vert_{q=\xi}$, i.e:
$$\check{\mathcal{R}}=\sum_{i=0}^{\ell-1}\frac{\{1\}^{i}e_1^{i} \otimes f_1^{i}}{(i)_{\xi}!} \sum_{j=0}^{1}\frac{(-\{1\})^{j}e_3^{j} \otimes f_3^{j}}{(j)_{\xi}!} \sum_{k=0}^{1}\frac{(-\{1\})^{k}e_2^{k} \otimes f_2^{k}}{(k)_{\xi}!} \in \mathcal{U}^{H} \otimes \mathcal{U}^{H}.$$

\begin{The} \label{ptur}
The operator $\mathcal{R}=\check{\mathcal{R}}\mathcal{K}$ led to a braiding $\{c_{V,W}\}$ in the category $\mathscr{C}^{H}$ where $c_{V,W}: V \otimes W \rightarrow W \otimes V$ is determined by $v \otimes w \mapsto \tau(\mathcal{R}(v \otimes w))$. Here $\tau: V \otimes W \rightarrow W \otimes V, v \otimes w \mapsto (-1)^{\deg v \deg w}w \otimes v$.
\end{The}
\begin{proof}
It is sufficient to prove that the operator $\mathcal{R}$ satisfies
\begin{equation} \label{dkcm}
\Delta \otimes \Id(\mathcal{R})=\mathcal{R}_{13}\mathcal{R}_{23}, \Id \otimes \Delta(\mathcal{R})=\mathcal{R}_{13}\mathcal{R}_{12}, \mathcal{R}\Delta^{op}(x)=\Delta(x)\mathcal{R}
\end{equation}
for all $x \in \mathcal{U}^{H}$. 

Let $\chi_{q}: \mathcal{U}_{q}\mathfrak{sl}(2|1)\otimes \mathcal{U}_{q}\mathfrak{sl}(2|1)\rightarrow \mathcal{U}_{q}\mathfrak{sl}(2|1) \otimes \mathcal{U}_{q}\mathfrak{sl}(2|1)$ be the automorphism determined by $x \otimes y \mapsto \mathcal{K}_{q}(x \otimes y)\mathcal{K}_{q}^{-1}$, this one induces an automorphism $\chi_{\xi}: \mathcal{U}^{H} \otimes \mathcal{U}^{H}\rightarrow \mathcal{U}^{H} \otimes \mathcal{U}^{H}$. We consider the element $\check{\mathcal{R}}^{<}$ of $\mathcal{U}_q\mathfrak{sl}(2|1) \otimes \mathcal{U}_q\mathfrak{sl}(2|1)$, Proposition \ref{mdR} implies the relations
\begin{align} \label{qhR<}
&\Delta \otimes \Id(\check{\mathcal{R}})=\check{\mathcal{R}}_{13}\left(\chi_{\xi}\right)_{13}(\check{\mathcal{R}}_{23}), \\
&\Id \otimes \Delta(\check{\mathcal{R}})=\check{\mathcal{R}}_{13}\left(\chi_{\xi}\right)_{13}(\check{\mathcal{R}}_{12}), \\
&\check{\mathcal{R}}\left(\chi_{\xi}\right)(\Delta^{op}(x))=\Delta(x)\check{\mathcal{R}} \  \text{for all} \ x \in \mathcal{U}^{H}.
\end{align}
We will prove the equality (\ref{qhR<}), and that the other two are similar. From the first 
equality of the Proposition \ref{mdR}, we deduce that $(\Delta \otimes \Id)(\check{\mathcal{R}}^{<}\mathcal{K}_q)=\check{\mathcal{R}}^{<}_{13}(\mathcal{K}_q)_{13}\check{\mathcal{R}}^{<}_{23}(\mathcal{K}_q)_{23}$. The term in the left of this equality is equal to $(\Delta \otimes \Id)(\check{\mathcal{R}}^{<})(\Delta \otimes \Id)(\mathcal{K}_q)=\Delta \otimes \Id(\check{\mathcal{R}}^{<})(\mathcal{K}_q)_{13}(\mathcal{K}_q)_{23}$. 
The right one is equal to $\check{\mathcal{R}}^{<}_{13}(\mathcal{K}_q)_{13}\check{\mathcal{R}}^{<}_{23}(\mathcal{K}_q)_{23}=\check{\mathcal{R}}^{<}_{13}\left(\chi_{q}\right)_{13}(\check{\mathcal{R}}^{<}_{23})(\mathcal{K}_q)_{13}(\mathcal{K}_q)_{23}$. Now because $\mathcal{K}_q$ is invertible, the result is $\Delta \otimes \Id(\check{\mathcal{R}}^{<})= \check{\mathcal{R}}^{<}_{13}\left(\chi_{q}\right)_{13}(\check{\mathcal{R}}^{<}_{23})$.

The element $\check{\mathcal{R}}^{<}$ has no pole when $q$ is a root of unity of order $\ell$. Hence we can specialize this relation at $q=\xi$ and $\Delta \otimes \Id(\check{\mathcal{R}})=\check{\mathcal{R}}_{13}\left(\chi_{\xi}\right)_{13}(\check{\mathcal{R}}_{23})$. Finally, as operators on $V_1 \otimes V_2 \otimes V_3$ in which $V_1, V_2, V_3 \in \mathscr{C}^{H}$, equation (\ref{ptlh}) implies that 
\begin{align*}
\Delta \otimes \Id(\mathcal{R})&=(\Delta \otimes \Id)(\check{\mathcal{R}})(\Delta \otimes \Id)(\mathcal{K})\\
&=\check{\mathcal{R}}_{13}\left(\chi_{\xi}\right)_{13}(\check{\mathcal{R}}_{23}) \mathcal{K}_{13}\mathcal{K}_{23}\\
&=\check{\mathcal{R}}_{13}\mathcal{K}_{13}\check{\mathcal{R}}_{23}\mathcal{K}_{13}^{-1}\mathcal{K}_{13}\mathcal{K}_{23}\\
&=\check{\mathcal{R}}_{13}\mathcal{K}_{13} \check{\mathcal{R}}_{23}\mathcal{K}_{23}\\
&=\mathcal{R}_{13}\mathcal{R}_{23}.
\end{align*}
Thus the relations of equation (\ref{dkcm}) hold.
\end{proof}

The category $\mathscr{C}^{H}$ is pivotal and braided with the braiding $c_{V,W}: V \otimes W \rightarrow W \otimes V, v \otimes w \mapsto \tau \circ \mathcal{R}(v \otimes w)$ where $V, W \in \mathscr{C}^{H}$.
\subsubsection{Ribbon category $\mathscr{C}^{H}$}
To prove the next proposition we will use the semi-simplicity of $\mathscr{C}_{g}$ ($g \in \mathit{G}\backslash \mathit{G}_{s}$) which is proven later in Theorem \ref{dl2}.
\begin{Pro} \label{twist}
The family of isomorphisms $\theta_V: V \rightarrow V$ determined by $\theta_V=(\Id_{V}\otimes \tev_{V})(c_{V,V}\otimes \Id_{V^{*}})(\Id_{V}\otimes \coev_{V}), V \in \mathscr{C}^{H}$ is a twist. That is $\theta_{V}=\theta_{V}^{'} \ \forall V \in \mathscr{C}^{H}$ where $\theta_{V}^{'}=(\ev_V \otimes \Id_V)(\Id_{V^{*}} \otimes c_{V,V})(\tcoev_V \otimes \Id_V)$.
\end{Pro}
\begin{proof}
	Firstly, if $V$ is a typical module of highest weight $\mu=(\mu_{1},\mu_{2}),
	V \in \mathscr{C}_{g}^{H}, g \in \mathit{G} \backslash \mathit{G}_{s}$, we have $\theta_{V}^{'}=(\ev_{V}\otimes \Id_{V})(\Id_{V^{*}}\otimes c_{V,V})(\tcoev_{V}\otimes \Id_{V})=X_{1}X_{2}X_{3}$.\\
	We use the vector of lowest weight $(\mu_{1}-2\ell+2, \mu_{2}+\ell)$ of $V, \ w_{1,1,\ell-1}:=w_{\infty},$ to calculate.
	\begin{align*}
	X_{3}(w_{\infty})&=\sum_{\rho, \sigma, p}(-1)^{\rho + \sigma}w_{\rho, \sigma, p}^{*} \otimes \phi_{0}^{-1}w_{\rho, \sigma, p} \otimes w_{\infty}\\
	&=\sum_{\rho, \sigma, p}(-1)^{\rho + \sigma}\xi^{\ell \mu_{1}+2\mu_{2}+2\sigma+2p}w_{\rho, \sigma, p}^{*} \otimes w_{\rho, \sigma, p} \otimes w_{\infty}.
	\end{align*}
		$X_{2}X_{3}(w_{\infty})=\sum_{\rho, \sigma, p}(-1)^{\rho + \sigma}\xi^{\ell \mu_{1}+2\mu_{2}+2\sigma+2p}w_{\rho, \sigma, p}^{*} \otimes (\tau \circ \mathcal{R})(w_{\rho, \sigma, p} \otimes w_{\infty})$.
		\begin{align*}
	\mathcal{K}(w_{\rho, \sigma, p} \otimes w_{\infty})&=\xi^{-h_{1}\otimes h_{2}-h_{2}\otimes h_{1}-2h_{2}\otimes h_{2}}w_{\rho, \sigma, p} \otimes w_{\infty}\\
	&=\xi^{-\mu_{1}(\mu_{2}+\sigma+p+\ell)-\mu_{2}(\mu_{1}+2\mu_{2}+\sigma+\rho+2)-2(\sigma+p)}w_{\rho, \sigma, p} \otimes w_{\infty}.
	  \end{align*}
		\begin{align*}
		&X_{2}X_{3}(w_{\infty})\\
		&=\sum_{\rho, \sigma, p}(-1)^{\rho + \sigma}\xi^{\ell \mu_{1}+2\mu_{2}}\xi^{-\mu_{1}(\mu_{2}+\sigma+p+\ell )-\mu_{2}(\mu_{1}+2\mu_{2}+\sigma+\rho+2)}w_{\rho, \sigma, p}^{*} \otimes w_{\infty} \otimes w_{\rho, \sigma, p}\\
		&=\sum_{\rho, \sigma, p}(-1)^{\rho + \sigma}\xi^{-\mu_{1}(\mu_{2}+\sigma+p)-\mu_{2}(\mu_{1}+2\mu_{2}+\sigma+\rho)}w_{\rho, \sigma, p}^{*} \otimes w_{\infty} \otimes w_{\rho, \sigma, p}.
		\end{align*}
So
		\begin{align*}
		X_{1}X_{2}X_{3}(w_{\infty})&=\sum_{\rho, \sigma, p}(-1)^{\rho+\sigma}\xi^{-\mu_{1}(\mu_{2}+\sigma+p)-\mu_{2}(\mu_{1}+2\mu_{2}+\sigma+\rho)}w_{\rho, \sigma, p}^{*}(w_{\infty}) \otimes w_{\rho, \sigma, p}\\
		&=\xi^{-\mu_{1}(\mu_{2}+\ell )-\mu_{2}(\mu_{1}+2\mu_{2}+2)} w_{\infty}.
		\end{align*}
Secondly, we have $$\theta_{V}=(\Id_{V}\otimes \tev_{V})(c_{V,V}\otimes \Id_{V^{*}})(\Id_{V}\otimes \coev_{V})=Y_{1}Y_{2}Y_{3}.$$
$Y_{3}(w_{0,0,0})=\sum_{\rho, \sigma, p}w_{0,0,0}\otimes w_{\rho, \sigma, p} \otimes w_{\rho, \sigma, p}^{*}, \\
Y_{2}Y_{3}(w_{0,0,0})=\sum_{\rho, \sigma, p}(\tau \circ \mathcal{R})(w_{0,0,0}\otimes w_{\rho, \sigma, p}) \otimes w_{\rho, \sigma, p}^{*}$ where \\ 
$\mathcal{K}(w_{0,0,0}\otimes w_{\rho, \sigma, p})=\xi^{-\mu_{1}(\mu_{2}+\sigma+p)-\mu_{2}(\mu_{1}+\rho-\sigma -2p)-2\mu_{2}(\mu_{2}+\sigma+p)}w_{0,0,0}\otimes w_{\rho, \sigma, p}$ and\\
$\mathcal{R}(w_{0,0,0}\otimes w_{\rho, \sigma, p})=\xi^{-\mu_{1}(\mu_{2}+\sigma+p)-\mu_{2}(\mu_{1}+\rho-\sigma -2p)-2\mu_{2}(\mu_{2}+\sigma+p)}w_{0,0,0}\otimes w_{\rho, \sigma, p}$.\\
$Y_{2}Y_{3}(w_{0,0,0})=\sum_{\rho, \sigma, p}\xi^{-\mu_{1}(\mu_{2}+\sigma+p)-\mu_{2}(\mu_{1}+\rho-\sigma -2p)-2\mu_{2}(\mu_{2}+\sigma+p)}w_{\rho, \sigma, p} \otimes w_{0,0,0}\otimes w_{\rho, \sigma, p}^{*}.$ 
	\begin{align*}
&Y_{1}Y_{2}Y_{3}(w_{0,0,0})\\
&=\sum_{\rho, \sigma, p}\xi^{-\mu_{1}(\mu_{2}+\sigma+p)-\mu_{2}(\mu_{1}+\rho-\sigma -2p)-2\mu_{2}(\mu_{2}+\sigma+p)}w_{\rho, \sigma, p} \otimes w_{\rho, \sigma, p}^{*}((-1)^{\rho+\sigma}\phi_{0} w_{0,0,0})\\
&=\sum_{\rho, \sigma, p}(-1)^{\rho+\sigma}\xi^{-\mu_{1}(\mu_{2}+\sigma+p)-\mu_{2}(\mu_{1}+\rho-\sigma -2p)-2\mu_{2}(\mu_{2}+\sigma+p)}w_{\rho, \sigma, p} \otimes w_{\rho, \sigma, p}^{*}(\xi^{-\ell\mu_{1}-2\mu_{2}}w_{0,0,0})\\
&=\xi^{-2\mu_{1}\mu_{2}-2\mu_{2}^{2}-2\mu_{2}-\ell\mu_{1}}w_{0,0,0} =\xi^{-\mu_{1}(\mu_{2}+\ell)-\mu_{2}(\mu_{1}+2\mu_{2}+2)}w_{0,0,0}.
	\end{align*}
	We can deduce that $\theta_{V}=\theta_{V}^{'}$ for every typical module $V$ with highest weight $\mu=(\mu_{1},\mu_{2}), V \in \mathscr{C}_{g}^{H}, g \in \mathit{G}\backslash \mathit{G}_{s}$. Note that the calculation does not change if we reverse the parity of vectors. So we have the affirmation for a semi-simple module in degree $g \in \mathit{G}\backslash \mathit{G}_{s}$. 
Let a module $W \in \mathscr{C}_{g}^{H}, g \in \mathit{G}$. By Theorem \ref{dl2} it exists $h \in \mathit{G}$ such that $\mathscr{C}_{h}^{H}, \mathscr{C}_{g+h}^{H}$ are semi-simple. For a module $V \in \mathscr{C}_{h}^{H}$ we have $W \otimes V \in \mathscr{C}_{g+h}^{H}$ is semi-simple.

Because $\theta_{W \otimes V}=(\theta_{W} \otimes \theta_{V})c_{V,W}c_{W,V}= \theta_{W \otimes V}^{'}=(\theta_{W}^{'} \otimes \theta_{V}^{'})c_{V,W}c_{W,V}$ and $\theta_{V}=\theta_{V}^{'}$, we deduce that $\theta_{W}=\theta_{W}^{'} \, \forall W \in \mathscr{C}^{H}$. i.e. the family $\theta_{V}$ is a twist.
	\end{proof}

\begin{Lem} \label{kihieu alpha}
Let $\mu =(\mu_1,\mu_2)\in \mathbb{C}\times \mathbb{C}$, then the value of the twist $\theta_{V_{\mu}}$ on a simple module $V_{\mu}$ with highest weight $\mu$ is $\xi^{-\ell\mu_1-2\mu_2(1+\mu_1+\mu_2)}\Id_{V_{\mu}}$. That is,
$$\theta_{V_{\mu}}=\xi^{-\ell\mu_1-2\mu_2(1+\mu_1+\mu_2)}\Id_{V_{\mu}}=-\xi^{-2(\alpha_{2}^{2}+\alpha_{1}\alpha_{2})}\Id_{V_{\mu}}$$
where $\alpha=(\alpha_{1}, \alpha_{2})=(\mu_{1}-\ell+1, \mu_{2}+\frac{\ell}{2})$.
\end{Lem}
	\begin{proof}
	By the proof of Proposition \ref{twist}, $\theta_{V_{\mu_{1},\mu_{2}}}=\xi^{-\ell\mu_1-2\mu_2(1+\mu_1+\mu_2)}\Id_{V_{\mu}}$. 
	\end{proof}
	
	The category $\mathscr{C}^{H}$ is a braided pivotal category with a twist, i.e. $\mathscr{C}^{H}$ is a ribbon category.

	Let $\mathcal{T}$ be the ribbon category  of $\mathscr{C}$-colored oriented ribbon graphs in the sense of Turaev \cite{Tura94}. The set of morphisms $\mathcal{T}(((V_1,\pm),..., (V_n,\pm)),\\ ((W_1,\pm),...,(W_n,\pm)))$ is a space of linear combinations of $\mathscr{C}$-colored ribbon graphs. The ribbon Reshetikhin-Turaev  functor $F:\mathcal{T} \rightarrow \mathscr{C}^{H}$ is defined by the Penrose graphical calculus. 
	\begin{Def}
If $T \in \mathcal{T}((V_{\mu},+), (V_{\mu},+))$ where $V_{\mu}$ is a simple weight module of $\mathcal{U}_{\xi}^{H}\mathfrak{sl}(2|1)$, then $F(T)=x.\Id_{V_{\mu}}\in \End_{\mathcal{U}_{\xi}^{H}\mathfrak{sl}(2|1)}(V_{\mu})$ for $x\in \mathbb{C}$. We define the bracket of $T$ by $\langle T \rangle=x$.
For example, if $V_{\mu},V_{\mu^{'}} \in \mathscr{C}^{H}$, we define $S'(V_{\mu},V_{\mu^{'}}) =\left< \epsh{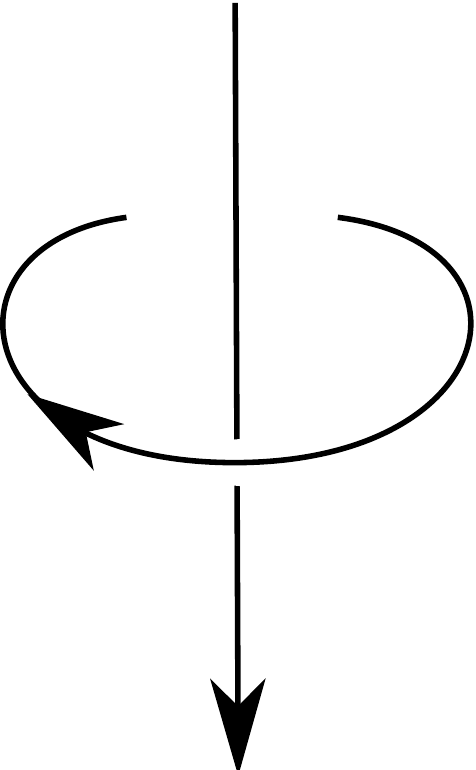}{7ex}
      \put(-8,-7){\ms{V_{\mu}}}
			\put(-11,21){\ms{V_{\mu^{'}}}} \right>.$
	\end{Def}

	We write $S'(\mu, \mu^{'})$ for $S'(V_{\mu},V_{\mu^{'}})$.

Another example is the bracket of the twist $\left<\epsh{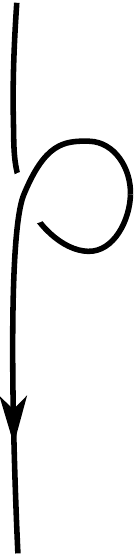}{7ex}
      \put(-6,-18){\ms{V_{\mu}}}\right>
=-\xi^{-2(\alpha_{2}^{2}+\alpha_{1}\alpha_{2})}$, $(\alpha_{1}, \alpha_{2})=(\mu_{1}-\ell+1, \mu_{2}+\frac{\ell}{2})$.

\begin{Pro} \label{tinh S'}
Let $V=V_{\mu}$ be a typical module, $V^{'}=V_{\mu^{'}}$ be a simple module, then
$$S^{'}(\mu, \mu^{'})=\xi^{-4\alpha_{2}\alpha_{2}^{'}-2(\alpha_{2}\alpha_{1}^{'}+\alpha_{1}\alpha_{2}^{'})} \frac{\{\ell\alpha_1^{'}\}\{\alpha_2^{'}\}\{\alpha_2^{'}+\alpha_1^{'}\}}{\{\alpha_1^{'}\}}$$
where $\alpha=(\alpha_{1}, \alpha_{2})=(\mu_{1}-\ell+1, \mu_{2}+\frac{\ell}{2}), \alpha^{'}=(\alpha_{1}^{'}, \alpha_{2}^{'})=(\mu_{1}^{'}-\ell+1, \mu_{2}^{'}+\frac{\ell}{2})$.
\end{Pro}

\begin{proof}
Let $S=S(\mu, \mu^{'}) \in \End_{\mathbb{C}}(V_{\mu_1^{'},\mu_2^{'}})$ be the endomorphism determined by the diagram $\epsh{fig4}{7ex}
      \put(-8,-7){\ms{V_{\mu}}}
			\put(-11,21){\ms{V_{\mu^{'}}}}\ .$ 
We have 
\begin{align*}
S'(\mu, \mu^{'})\Id_{V_{\mu^{'}}} &=(\Id_{V^{'}}\otimes \tev_{V})(c_{V,V^{'}}\otimes \Id_{V^{*}})(c_{V^{'},V}\otimes \Id_{V^{*}})(\Id_{V^{'}}\otimes \coev_{V})\\
		&=X_{1}X_{2}X_{3}X_{4}.
\end{align*} 
The definition gives us \\ $X_{4}(w_{0,0,0}^{'})=\sum_{\rho, \sigma, p}w_{0,0,0}^{'} \otimes w_{\rho, \sigma, p} \otimes w_{\rho, \sigma, p}^{*}$ and \\
$X_{3}X_{4}(w_{0,0,0}^{'})=\sum_{\rho, \sigma, p}(\tau \circ \mathcal{R})(w_{0,0,0}^{'} \otimes w_{\rho, \sigma, p})\otimes w_{\rho, \sigma, p}^{*}$.\\
$\mathcal{K}(w_{0,0,0}^{'} \otimes w_{\rho, \sigma, p})=\xi^{-\mu_{1}^{'}(\mu_{2}+\sigma +p)-\mu_{2}^{'}(\mu_{1}+\rho-\sigma -2p)-2\mu_{2}^{'}(\mu_{2}+\sigma +p)}w_{0,0,0}^{'} \otimes w_{\rho, \sigma, p}$.\\
$\mathcal{R}(w_{0,0,0}^{'} \otimes w_{\rho, \sigma, p})=\xi^{-\mu_{1}^{'}(\mu_{2}+\sigma +p)-\mu_{2}^{'}(\mu_{1}+\rho-\sigma -2p)-2\mu_{2}^{'}(\mu_{2}+\sigma +p)}w_{0,0,0}^{'} \otimes w_{\rho, \sigma, p}$.\\
So $X_{3}X_{4}(w_{0,0,0}^{'})=\sum_{\rho, \sigma, p}\xi^{-\mu_{1}^{'}(\mu_{2}+\sigma +p)-\mu_{2}^{'}(\mu_{1}+\rho-\sigma -2p)-2\mu_{2}^{'}(\mu_{2}+\sigma +p)}w_{\rho, \sigma, p} \otimes w_{0,0,0}^{'} \otimes w_{\rho, \sigma, p}^{*}$.\\
$X_{2}X_{3}X_{4}(w_{0,0,0}^{'})=\sum_{\rho, \sigma, p}\xi^{-\mu_{1}^{'}(\mu_{2}+\sigma +p)-\mu_{2}^{'}(\mu_{1}+\rho-\sigma -2p)-2\mu_{2}^{'}(\mu_{2}+\sigma +p)}(\tau \circ \mathcal{R})(w_{\rho, \sigma, p} \otimes w_{0,0,0}^{'}) \otimes w_{\rho, \sigma, p}^{*}$.\\
	Furthermore, the element $(\mathcal{\check{R}}-1)(w_{\rho, \sigma, p} \otimes w_{0,0,0}^{'}) \in V_{\mu_1, \mu_2} \otimes V_{\mu_{1}^{'},\mu_{2}^{'}}$ is a sum of vectors of the form $v^{'}\otimes w^{'}$ where $w^{'}$ is a weight vector of $V_{\mu_{1}^{'},\mu_{2}^{'}}$ and $v^{'}$ is a weight vector of $V_{\mu_1, \mu_2}$ which has a higher weight than $w_{\rho, \sigma, p}$.\\
	$X_{2}X_{3}X_{4}(w_{0,0,0}^{'})=\sum_{\rho, \sigma, p}(\xi^{-\mu_{1}^{'}(\mu_{2}+\sigma +p)-\mu_{2}^{'}(\mu_{1}+\rho-\sigma -2p)-2\mu_{2}^{'}(\mu_{2}+\sigma +p)}w_{0,0,0}^{'} \otimes w_{\rho, \sigma, p}\otimes w_{\rho, \sigma, p}^{*}+\sum_{k}w_{k}^{'}\otimes v_{k}^{'} \otimes z_{k}).$
	\begin{align*}
	&X_{1}X_{2}X_{3}X_{4}(w_{0,0,0}^{'})\\
	&=\sum_{\rho, \sigma, p}\xi^{-\mu_{1}^{'}(\mu_{2}+\sigma +p)-\mu_{2}^{'}(\mu_{1}+\rho-\sigma -2p)-2\mu_{2}^{'}(\mu_{2}+\sigma +p)}w_{0,0,0}^{'} \otimes (-1)^{\rho+\sigma}w_{\rho, \sigma, p}^{*}(\phi_{0} w_{\rho, \sigma, p})\\
	&=\sum_{\rho, \sigma, p}\xi^{-\mu_{1}^{'}(\mu_{2}+\sigma +p)-\mu_{2}^{'}(\mu_{1}+\rho-\sigma -2p)-2\mu_{2}^{'}(\mu_{2}+\sigma +p)-\ell\mu_{1}-2(\mu_{2}+\sigma+p)}w_{0,0,0}^{'} \\
	&=\xi^{-(2\mu_2+\mu_1+1)(2\mu_2^{'}+\mu_1^{'}+1)+(\mu_1+1)(\mu_1^{'}+1)-\ell(\mu_1^{'}+\mu_1+1)} \frac{\{\ell(\mu_1^{'}+1)\}\{\mu_2^{'}\}\{\mu_2^{'}+\mu_1^{'}+1\}}{\{\mu_1^{'}+1\}}\\
	\ &w_{0,0,0}^{'}\\
	&=\xi^{-4\alpha_{2}\alpha_{2}^{'}-2(\alpha_{2}\alpha_{1}^{'}+\alpha_{1}\alpha_{2}^{'})}\frac{\{\ell\alpha_1^{'}\}\{\alpha_2^{'}\}\{\alpha_2^{'}+\alpha_1^{'}\}}{\{\alpha_1^{'}\}}w_{0,0,0}^{'}.
	\end{align*}

By the definition $S(\mu, \mu^{'})(w_{0,0,0}^{'} )=S^{'}(\mu, \mu^{'})w_{0,0,0}^{'} $, we deduce the proposition. 
\end{proof}

\begin{Def} \label{dn d}
If $\mu =(\mu_1,\mu_2)\in \left(\mathbb{C} \backslash \frac{1}{2}\mathbb{Z} \cup (-1+\frac{\ell}{2}\mathbb{Z})\right) \times \mathbb{C}\backslash \frac{\ell}{2}\mathbb{Z}$ and $\mu_{2}+\mu_{1}+1 \in \mathbb{C}\backslash \frac{\ell}{2}\mathbb{Z}$, we define $$d(\mu)=\frac{\{\mu_1+1\}}{\ell \{\ell \mu_{1}\}\{\mu_2\}\{\mu_2+\mu_1+1\}}=\frac{\{\alpha_{1}\}}{\ell \{\ell \alpha_{1}\}\{\alpha_{2}\}\{\alpha_{1}+\alpha_{2}\}},$$
so there is a symmetry
$$d(\mu^{'})S'(\mu, \mu^{'})=d(\mu)S'(\mu^{'}, \mu).$$ 
\end{Def}
\subsubsection{Semi-simplicity of category $\mathscr{C}^{H}$}
	
Remember that $\mathit{G}=\mathbb{C} / \mathbb{Z} \times \mathbb{C} / \mathbb{Z}$ and $\mathit{G}_{s}=\{\overline{g} \in \mathit{G} \ \textit{such that} \  \exists \ V \in \mathscr{C}_{g}^{H}\ \textit{simple of}\ \mathscr{C}^{H} \ \textit{and atypical} \}$.
\begin{Lem}
If $\mathscr{C}_{\overline{\mu}}^{H}$ is semi-simple, then a module $V \in \mathscr{C}_{\overline{\mu}}^{H}$ is determined up to an isomorphism and parity by its character.
\end{Lem}
	
	The above lemma and the character of representation $V_{\mu_{1}, \mu_{2}} \otimes V_{\mu_{1}^{'}, \mu_{2}^{'}}$ gives us the following theorem.
\begin{The} \label{dlnd}
Let $V_{\mu}, V_{\mu^{'}}$ be two typical modules. If $\overline{\mu+\mu^{'}} \notin \mathit{G}_{s}$ then
$V_{\mu_1,\mu_2} \otimes V_{\mu_1^{'},\mu_2^{'}}=\oplus_{k=0}^{\ell-1}(V_{\mu_1+\mu_1^{'}-2k,\mu_2+\mu_2^{'}+k}\oplus \overline{V}_{\mu_1+\mu_1^{'}-2k+1,\mu_2+\mu_2^{'}+k} 
\oplus \overline{V}_{\mu_1+\mu_1^{'}-2k,\mu_2+\mu_2^{'}+k+1} \oplus V_{\mu_1+\mu_1^{'}-2k-1,\mu_2+\mu_2^{'}+k+1})$ \vspace{12pt} where $\overline{V}$ is the module $V$ with opposite parity.
\end{The}
\begin{proof}
According to the formula (\ref{dac so sl21}), we have 
\begin{align*}
&\chi_{V_{\mu_1, \mu_2} \otimes V_{\mu_{1}^{'}, \mu_{2}^{'}}}=\chi_{V_{\mu_1, \mu_2}} \chi_{V_{\mu_{1}^{'}, \mu_{2}^{'}}}\\
&= X_1^{\mu_1+\mu_{1}^{'}}X_2^{\mu_2 + \mu_{2}^{'}}\frac{1-x^{\ell}}{1-x}(1+X_1)(1+X_{1}x)\sum_{k=0}^{\ell-1}{(X_{1}^{-2}X_{2})^{k}}(1+X_{1}+X_{2}+X_1^{-1}X_2) \\
&=\frac{1-x^{\ell}}{1-x}(1+X_1)(1+X_{1}x)
\sum_{k=0}^{\ell-1}{X_1^{\mu_1+\mu_{1}^{'}-2k}X_2^{\mu_2 + \mu_{2}^{'}+k}}+X_1^{\mu_1+\mu_{1}^{'}-2k+1}X_2^{\mu_2 + \mu_{2}^{'}+k} \\
&+X_1^{\mu_1+\mu_{1}^{'}-2k}X_2^{\mu_2 + \mu_{2}^{'}+k+1}+X_1^{\mu_1+\mu_{1}^{'}-2k-1}X_2^{\mu_2 + \mu_{2}^{'}+k+1}.
\end{align*}
The analysis of parity of highest weight vectors can be used to conclude.
\end{proof}
\begin{rmq} Not all terms in the decomposition of the above theorem are distinct.
\end{rmq}
We defined a graduation $\mathscr{C}^{H}= \bigoplus_{\overline{\mu}\in \mathit{G}}\mathscr{C}_{\overline{\mu}}^{H}$. Let \texttt{Proj} be the subcategory of $\mathscr{C}^{H}$ containing projective modules, \texttt{Proj} is an ideal (see \cite{NgJkBp11}), i.e. \texttt{Proj} is closed under retracts and absorbent to the tensor product. We have the following proposition.
\begin{Pro} 
For $\overline{\mu}\in \mathit{G}$, the three conditions below are equivalent
\begin{enumerate}
	\item All the simple $\mathcal{U}_\xi\mathfrak{sl}(2|1)$-modules of $\mathscr{C}_{\overline{\mu}}$ are projective.
	\item The category $\mathscr{C}_{\overline{\mu}}$ is semi-simple.
	\item The $\mathbb{C}$-superalgebra of finite dimension $\mathcal{U}/(k_{1}^{\ell}-\xi^{\ell\overline{\mu_{1}}}, k_{2}^{\ell}-\xi^{\ell\overline{\mu_{2}}})$ is semi-simple where $\mathcal{U}=\mathcal{U}_\xi\mathfrak{sl}(2|1)/(e_{1}^{\ell},f_{1}^{\ell})$.
\end{enumerate}
\end{Pro}
\begin{proof}
The equivalence is classic knowing that $\mathscr{C}_{\overline{\mu}}$ is also a category of the $\mathcal{U}/(k_{1}^{\ell}-\xi^{\ell\overline{\mu_{1}}}, k_{2}^{\ell}-\xi^{\ell\overline{\mu_{2}}})$-modules. 
\end{proof}

\begin{The} \label{dl2}
\begin{enumerate}
	\item If $\overline{\mu}\in \mathit{G} \backslash \mathit{G}_{s}$ then $\mathscr{C}_{\overline{\mu}}^{H}$ is semi-simple. 
	\item A typical $\mathcal{U}_{\xi}^{H}\mathfrak{sl}(2|1)$-module is projective.
\end{enumerate}
\end{The}
We select and fix a $\overline{\mu}\in \mathit{G} \backslash \mathit{G}_{s}$, note $\mu_{i}=(\mu_{1}+i_{1}, \mu_{2}+i_{2}) \in \overline{\mu}, \ i_{1},i_{2}=0,1,...,\ell-1$, that is $\mu_{i} \in \{(\mu_1+i_{1}, \mu_2+i_{2}):\ i_{1},i_{2} =0,1,...,\ell-1 \}$. We have the two following lemmas. 
\begin{Lem} \label{ttai zij}
For all $\mu_{i}, \mu_{j} \in \overline{\mu}: \mu_{i} \ne \mu_{j}$ there exists $z_{ij} \in \mathcal{Z}$ such that $\chi_{\mu_{i}}(z_{ij}) \ne \chi_{\mu_{j}}(z_{ij})$ where $\chi_{\mu_{i}}(z_{ij}) \in \mathbb{C}$ is defined by $\rho_{\mu_i}(z_{ij})=\chi_{\mu_{i}}(z_{ij})\Id_{V_{\mu_i}}$.
\end{Lem}
\begin{proof}
We consider $\mu=(\mu_{1},\mu_{2}), \mu^{'}=(\mu_{1} +k,\mu_{2}+m) \ k,m=0,1,...,\ell-1$.
We suppose that $\forall \ z \in \mathcal{Z}: \chi_{\mu}(z) = \chi_{\mu^{'}}(z)$.
Consider the central elements $C_{p}$ where $p \in \mathbb{Z}$ (see \cite{BaDaMb97}). We have 
\begin{align*}
&\chi_{\mu}(C_{p})=(\xi-\xi^{-1})^{2}\xi^{(2p-1)(\mu_{1}+2\mu_{2})}[\mu_{2}][\mu_{2}+\mu_{1}+1],\\
&\chi_{\mu^{'}}(C_{p})=(\xi-\xi^{-1})^{2}\xi^{(2p-1)(\mu_{1}+2\mu_{2}+k+2m)}[\mu_{2}+m][\mu_{2}+\mu_{1}+k+m+1]. 
\end{align*}
Because $\chi_{\mu}(C_{p})=\chi_{\mu^{'}}(C_{p})$ and $[\mu_{2}][\mu_{2}+\mu_{1}+1] \ne 0$, we deduce that 
\begin{equation*}
\begin{cases}
\frac{\chi_{\mu}(C_{p+1})}{\chi_{\mu}(C_{p})}=\frac{\chi_{\mu^{'}}(C_{p+1})}{\chi_{\mu^{'}}(C_{p})}\\
\chi_{\mu}(C_{p})=\chi_{\mu^{'}}(C_{p}).
\end{cases} 
\end{equation*}
This is equivalent to
\begin{equation*}
\begin{cases}
\xi^{2(\mu_{1}+2\mu_{2})}=\xi^{2(\mu_{1}+2\mu_{2}+k+2m)}\\
\xi^{(2p-1)(\mu_{1}+2\mu_{2})}[\mu_{2}][\mu_{2}+\mu_{1}+1]=\xi^{(2p-1)(\mu_{1}+2\mu_{2}+k+2m)}[\mu_{2}+m][\mu_{2}+\mu_{1}+k+m+1],
\end{cases}
\end{equation*}
which implies 
\begin{equation} \label{mk}
2(k+2m)=0 \ (\text{modulo} \ \ell \mathbb{Z})
\end{equation} and
\begin{equation} \label{mocmk}
[\mu_{2}][\mu_{2}+\mu_{1}+1]=\xi^{k+2m}[\mu_{2}+m][\mu_{2}+\mu_{1}+k+m+1].
\end{equation}

	Because $\ell$ odd, \eqref{mk} implies $k+2m=0 \ (\text{modulo} \ \ell \mathbb{Z})$ $\Leftrightarrow k+m=-m \ (\text{modulo} \ \ell \mathbb{Z})$.
On the other hand, \eqref{mocmk} is equivalent to $[a][b]=[a+m][b-m]\Leftrightarrow-[a-b+m][m]=0\Leftrightarrow [-\mu_{1}-1+m][m]=0 \Rightarrow m=0$ where $a=\mu_{2}, b=\mu_{1}+\mu_{2}+1$.
Because $m=0$, we have $k=0 \ (\text{modulo} \ \ell\mathbb{Z})\Rightarrow k=0$.
\end{proof}
	\begin{Lem} \label{ttai z}
	Let $\mathcal{V}$ be a vector space over $\mathbb{C}$, $I$ be a finite set and consider a family of $\mathbb{C}$-linear functions $a_{i}: \mathcal{V} \rightarrow \mathbb{C}, \ i \in I$. If for all $i \neq j \ \exists \ u_{ij} \in \mathcal{V}$ such that $a_{i}(u_{ij}) \neq a_{j}(u_{ij})$, then it exists $u_{0} \in \mathcal{V}$ such that $\forall \ i\neq j \ a_{i}(u_{0}) \neq a_{j}(u_{0})$.
	\end{Lem}
	\begin{proof}
	We set $u=\sum_{i \ne j}x_{ij}u_{ij} \in \mathcal{V}$ with $x_{ij} \in \mathbb{C}, \ i,j \in I$. We note $x=(x_{ij}) \in \mathbb{C}^{N}$. We consider the set $X=\{x \in \mathbb{C}^{N} \ \exists i\ne j \ a_{i}(u) =a_{j}(u) \}=\{x \in \mathbb{C}^{N}: \ \sum_{i \ne j}(a_{i}(u_{ij}) - a_{j}(u_{ij}))x_{ij}=0 \}$, this is a finite reunion of hyperplanes of $\mathbb{C}^{N}$. This proves that $\exists x \notin X$ and this $x$ does not have the above property. That is, it exists $u_{0} \in \mathcal{V}$ such that $a_{i}(u_{0}) \neq a_{j}(u_{0})$ for all $i\ne j$.
	\end{proof}
	
	Now we change the basis of module $V_{\mu}$ as follows. We set
\begin{equation*}
	w_{\rho, \sigma, p}^{'}=
	\begin{cases}
	w_{\rho, \sigma, p} \ &\text{if} \ \rho=\sigma=0,1\\
	f_{1}^{p}w_{1,0,0} \ &\text{if} \ \rho=1, \sigma=0\\
	e_{1}^{(\ell-1)-p}w_{0,1,r-1} \ &\text{if} \ \rho=0, \sigma=1
	\end{cases}
\end{equation*}
	where $p=0,...,\ell-1$.
For the basis $\{w_{\rho, \sigma, p}^{'}\}$ we have the actions
\begin{align*}
&k_{1}w_{\rho, \sigma, p}^{'}=\xi^{\mu_{1}+\rho-\sigma-2p}w_{\rho, \sigma, p}^{'},\\
&k_{2}w_{\rho, \sigma,p}^{'}=\xi^{\mu_{2}+\sigma+p}w_{\rho, \sigma, p}^{'},\\
&e_{1}w_{1,0, p}^{'}=[p][\mu_{1}+2-p]w_{1,0, p-1}^{'},\\
&f_{1}w_{1,0, p}^{'}=w_{1,0, p+1}^{'},\\
&e_{1}w_{0,1, p}^{'}=w_{0,1, p-1}^{'},\\
&f_{1}w_{0,1, p}^{'}=[p+1][\mu_{1}-(p+1)]w_{0,1, p+1}^{'}.
\end{align*}
\begin{proof}[Proof of the theorem \ref{dl2}]
	
	We begin to show that $\mathscr{C}_{\overline{\mu}}$ is semi-simple.	
	We set $\mathcal{A}=\mathcal{U}/(k_{1}^{\ell}-\xi^{\ell\overline{\mu_{1}}}, k_{2}^{\ell}-\xi^{\ell\overline{\mu_{2}}})$. The density theorem implies that the application $\rho: \mathcal{A} \rightarrow \prod_{\mu_{i}}\End(V_{\mu_{i}})\cong \prod_{i=1}^{\ell^2}\mathcal{M}_{4\ell}(\mathbb{C})$ is surjective. We give here an elementary proof. 
	
	By Lemma \ref{ttai zij} and \ref{ttai z}, it exists an element $z \in \mathcal{Z}$ such that $\forall \ \mu_{i} \ne \mu_{j} \ \chi_{\mu_{i}}(z) \ne \chi_{\mu_{j}}(z)$ and we set $z_{i}=\chi_{\mu_{i}}(z) \ i=1,...,\ell^{2}$ and we introduce the ideal $J=\prod_{i=1}^{\ell^2}(z-z_i)\mathcal{A}$. \vspace{13pt}
		
	Firstly, we consider the representation $\rho: \mathcal{A}/(z-z_{i})\rightarrow \End_{\mathbb{C}}(V_{\mu_{i}})$. We will prove that $\rho$ is a surjection. We have $\End_{\mathbb{C}}(V_{\mu_{i}})\cong \mathcal{M}_{4\ell}(\mathbb{C})$.\vspace{13pt} We consider the elements $\Omega=\frac{k_{1}\xi+k_{1}^{-1}\xi^{-1}}{\{1\}^2}+f_{1}e_{1}=\frac{k_{1}\xi^{-1}+k_{1}^{-1}\xi}{\{1\}^2}+e_{1}f_{1}, c=k_{1}k_{2}^{2}, k_{1}$ in $\mathcal{U}_\xi\mathfrak{sl}(2|1)$.
	The actions of these elements on the basis $w_{\rho, \sigma, p}^{'}$ are defined by
\begin{align*}
&\Omega w_{0,0,p}^{'}=(\xi^{\mu_1+1}+\xi^{-\mu_{1}-1})w_{0,0,p}^{'}, \Omega w_{1,1,p}^{'}=(\xi^{\mu_1+1}+\xi^{-\mu_{1}-1})w_{1,1,p}^{'},\\
&\Omega w_{0,1,p}^{'}=\frac{\xi^{\mu_1}+\xi^{-\mu_{1}}}{\{1\}^{2}}w_{0,1,p}^{'}, \Omega w_{1,0,p}^{'}=\frac{\xi^{\mu_1+2}+\xi^{-\mu_{1}-2}}{\{1\}^{2}}w_{1,0,p}^{'},\\
& cw_{\rho, \sigma, p}^{'}=\xi^{\mu_{1}+2\mu_{2}+\rho + \sigma}w_{\rho, \sigma, p}^{'},\\
&k_{1}w_{\rho, \sigma, p}^{'}=\xi^{\mu_{1}+\rho-\sigma-2p}w_{\rho, \sigma, p}^{'}.\\
\end{align*}
	We now check that for all $w_{\rho, \sigma, m}^{'} \neq w_{\rho^{'}, \sigma^{'}, j}^{'} \ \exists \ u \in \{\Omega, c, k_{1}\}$ such that $\chi_{\mu_{i}}^{\rho, \sigma, m}(u) \neq \chi_{\mu_{i}}^{\rho^{'}, \sigma^{'}, j}(u)$ where $\rho(u)w_{\rho, \sigma, m}^{'}=\chi_{\mu_{i}}^{\rho, \sigma, m}(u)w_{\rho, \sigma, m}^{'} \  \rho,\sigma,\rho^{'},\sigma^{'} \in \{0,1\}, m,j \in \{0,...,\ell-1\}$. Indeed, if $\rho+\sigma \neq \rho^{'}+\sigma^{'}$ then we select $u=c$ and we have $c w_{\rho, \sigma, m}^{'} \neq c w_{\rho^{'}, \sigma^{'}, j}^{'}$. If $\rho+\sigma = \rho^{'}+\sigma^{'}$ then we consider two cases: if $(\rho,\sigma) \neq (\rho^{'},\sigma^{'})$ we select $u = \Omega$ and $\Omega w_{\rho, \sigma, m}^{'} \neq \Omega w_{\rho^{'}, \sigma^{'}, j}^{'}$; if $(\rho,\sigma) = (\rho^{'},\sigma^{'})$ we select $u=k_{1}$ and we have $k_{1} w_{\rho, \sigma, m}^{'} \neq k_{1} w_{\rho^{'}, \sigma^{'}, j}^{'}$ because $m \neq j$.

By Lemma \ref{ttai z} it exists a vector $u_{0} \in \mathbb{C}\langle \Omega, c, k_{1}\rangle$-space generated by the elements $\Omega, c, k_{1}$ such that $\chi_{\mu_{i}}^{\rho, \sigma, m}(u_{0}) \neq \chi_{\mu_{i}}^{\rho^{'}, \sigma^{'}, j}(u_{0})$ for all $w_{\rho, \sigma, m}^{'} \neq w_{\rho^{'}, \sigma^{'}, j}^{'}$.
The matrix $B$ determined by the application $\rho(u_{0})$ is a diagonal matrix which has $4\ell$ different eigenvalues. The image of the projection on the $i$-th eigenspace of $B$ is the matrix $E_{ii}, \ i=1,...,4\ell$. Hence the matrix $E_{ii}$ is in the image of $\rho$.
	
	For $i \in \{1,...,\ell^{2}\}, j \in \{1,...,4\ell\}$ we have $\rho(\mathcal{A}/(z-z_{i}))(v_{j}) \subset V_{\mu_i}$ (here we note $v_{j}$ the $j$-th vector of the basis) and $V_{\mu_i}$ is simple. Thus we deduce $\rho(\mathcal{A}/(z-z_{i}))(v_{j}) = V_{\mu_i}$. This proves that it exists $a_{0} \in \mathcal{A}/(z-z_{i})$ such that $\rho(a_{0})(v_{j}) = v_{n} \ \forall n \in \{1,...,4\ell\}$.  
	
	The endomorphism $\rho(a_{0})$ determines the matrix $(\rho(a_{0}))$ where $\rho(a_{0})_{jn}=1$. The matrix $E_{jn}$ is equal to $E_{jj}\rho(a_{0})_{jn}E_{nn}$, i.e. the matrix $E_{jn}$ is the image of an element in $\mathcal{A}/(z-z_{i})$. So the application $\rho$ is a surjection. 
This implies that the application $\prod_{i=1}^{\ell^{2}}\mathcal{A}/(z-z_{i}) \rightarrow \prod_{i=1}^{\ell^{2}}\mathcal{M}_{4\ell}(\mathbb{C})$ is surjective.
	
	Secondly, the composition $\prod_{i=1}^{\ell^{2}}\mathcal{A}/(z-z_{i}) \rightarrow \mathcal{A}/J \rightarrow \prod_{i=1}^{\ell^{2}}\mathcal{A}/(z-z_{i})$ is the identity. Thus, the application $\mathcal{A}/J \rightarrow \prod_{i=1}^{\ell^{2}}\mathcal{A}/(z-z_{i})$ is surjective. We deduce a series of surjections $\mathcal{A}\twoheadrightarrow \mathcal{A}/J \twoheadrightarrow \prod_{i=1}^{\ell^{2}}\mathcal{A}/(z-z_{i})\twoheadrightarrow \prod_{i=1}^{\ell^{2}}\mathcal{M}_{4\ell}(\mathbb{C})$, this sequence determines the surjection $\mathcal{A} \twoheadrightarrow \prod_{i=1}^{\ell^{2}}\mathcal{M}_{4\ell}(\mathbb{C})$. 
	
	Furthermore, the two algebras $\mathcal{A}$ and $\prod_{i=1}^{\ell^{2}}\mathcal{M}_{4\ell}(\mathbb{C})$ have the same dimension $16\ell^{4}$. This implies that this surjection is an isomorphism. This demonstrates that $\mathcal{A}$ is semi-simple. The category $\mathscr{C}_{\overline{\mu}}$ is also semi-simple.
	
	Now we prove that $\mathscr{C}_{\overline{\mu}}^{H}$ is semi-simple. Let $V^{H}$ be a module in $\mathscr{C}_{\overline{\mu}}^{H}$. Set $W=\Ker e_{1} \cap \Ker e_{2} \cap \Ker e_{3}$, it is a vector space of the highest weight vectors (the weights for $(h_1, h_2)$). We call $\{v_j\}_{j=1}^{n}$ a basis of weight vectors of $W$, we have $h_iv_j=\mu_{j}^{i}v_j, i=1, 2$.
	 So each $v_j$ generates a $\mathcal{U}_\xi^{H}\mathfrak{sl}(2|1)$-module $V_{j}$,
	 $$V_{j}=\mathcal{U}_\xi^{H}\mathfrak{sl}(2|1).v_j = \mathcal{U}_\xi\mathfrak{sl}(2|1).v_j=\mathcal{U}_{-}.v_j$$
	 where $\mathcal{U}_{-}= \Alg \langle f_1, f_2, f_3 \rangle \subset \mathcal{U}_\xi\mathfrak{sl}(2|1)$ and $\dim (\mathcal{U}_{-}) = 4 \ell$.
	 Thus $\dim (V_j)\leq 4 \ell$ and $V_j$ is simple (because there is no module in $\mathscr{C}_{\overline{\mu}}^{H}$ of dimension strictly between $0$ and $4 \ell$).
	 
	 Set $V^{'}=\sum_{i=1}^{n}V_i \subset V^{H}$. We can write $V^{H}=V^{'}\oplus V^{''}$ as a $\mathcal{U}_\xi\mathfrak{sl}(2|1)$-module. However $W \subset V^{'}$ which implies $V^{''}=0$ (because there is no highest weight vector in $V^{''}$) and $V^{H}=V^{'}=\sum_{i=1}^{n}V_i$. Because the $V_i$ are simple, so $V^{H} =\bigoplus_{i \in I}V_i$ where $I \subset \{1, ...,n\}$. Thus $V^{H}$ is semi-simple. 
	
	For the second assertion $(2)$, if $V \in \mathscr{C}_{\overline{\mu}}^{H}$ and $\mathscr{C}_{\overline{\mu}}^{H}$ is semi-simple, then $V$ is projective. If not, $(2)$ follows from $S'(V_{\mu},V) \neq 0$ where $V_{\mu}$ is any projective typical module which implies that $V$ is a direct factor of $V_{\mu} \otimes V \otimes V_{\mu}^{*} \in \texttt{Proj}$. This implies that $V$ is a projective module. 
\end{proof}
\section{Modified traces on the projective modules}
\subsection{Ambidextrous module}
For each object $V$ of the category $\mathscr{C}$ and any endomorphism $f$ of $V \otimes V$ set 
\begin{align*}
&\ptr_{R}(f)=(\Id_{V}\otimes \tev_{V})\circ (f \otimes \Id_{V^{*}})\circ (\Id_{V} \otimes \coev_{V}) \in \End(V),\\
&\ptr_{L}(f)=(\ev_{V}\otimes \Id_{V})\circ (\Id_{V^{*}} \otimes f)\circ (\tcoev_{V} \otimes \Id_{V}) \in \End(V).
\end{align*}
In the ribbon category $\mathscr{C}^{H}$ of nilpotent weight $\mathcal{U}_{\xi}^{H}\mathfrak{sl}(2|1)$-modules, we say that a module $V$ is {\em ambidextrous} if $V$ simple and $\ptr_{L}(f)=\ptr_{R}(f)$ for all $f \in \End(V \otimes V)$ (see \cite{NgBpVt09}).

\begin{The}
Each typical module $V_{\mu}$ of category $\mathscr{C}^{H}$ is an ambidextrous module.
\end{The}
\begin{proof}
	We will prove this theorem in two steps:\\
	Step 1. Proving the existence of two nonzero $\mathcal{U}_{\xi}^{H}\mathfrak{sl}(2|1)$-invariant vectors $x_{-}w_{+}$ and $x_{+}w_{-}$.\\
	Step 2. Applying Theorem 3.1.3 \cite{NgJkBp13} gives us the affirmation that $V_{\mu}$ is ambidextrous.

	Call $v_{+}, v_{+}^{'}$ the highest weight vectors of $V_{\mu}, V_{\mu}^{*}$ and $v_{-}, v_{-}^{'}$ the lowest weight vectors of $V_{\mu}, V_{\mu}^{*}$. Set $x_{-}=f_{2}f_{3}f_{1}^{\ell-1}, x_{+}=e_{2}e_{3}e_{1}^{\ell-1}, w_{+}=v_{+} \otimes v_{+}^{'}, w_{-}=v_{-} \otimes v_{-}^{'}$. We will prove that the two vectors $x_{-}w_{+}$ and $x_{+}w_{-}$ are $\mathcal{U}_{\xi}^{H}\mathfrak{sl}(2|1)$-invariant.

	We consider the actions of generator elements $e_{i},h_{i}, f_{i}$ on $x_{-}w_{+}$. The highest weight vector (resp. lowest) of $V_{\mu}$ is $v_{+}=w_{0,0,0}$ (resp. $v_{-}=w_{1,1,\ell-1})$. The highest weight vector  (resp. lowest) of $V_{\mu}^{*}$ is $v_{+}^{'}=w_{1,1,\ell-1}^{*}$ (resp. $v_{-}^{'}=w_{0,0,0}^{*}$).

	The weight of vector $w_{+}=v_{+} \otimes v_{+}^{'}$ is equal to the sum of the weights of $v_{+}$ and $v_{+}^{'}$. That is $\text{weight}(w_{+})=(\mu_{1}, \mu_{2})+(-\mu_{1}+2\ell -2, -\mu_{2}-\ell )=(2\ell -2, -\ell)$. Furthermore, $\text{weight}(x_{-}w_{+})=\text{weight}(f_{2}f_{3}f_{1}^{\ell-1}w_{+})=\text{weight}(f_{2}f_{1}f_{2}f_{1}^{\ell-1}w_{+})=-\ell \text{weight}(e_{1})-2\text{weight}(e_{2})+\text{weight}(w_{+})=-\ell (2,-1)-2(-1,0)+(2\ell -2, -\ell)=(0,0)$. It implies that $h_{i}x_{-}w_{+}=0$.

	We also have the relations below between the generator elements in $\mathcal{U}_{\xi}^{H}\mathfrak{sl}(2|1)$ (see (B1) \cite{BaDaMb97}):
\begin{align*}
&f_{1}f_{2}^{\rho}f_{3}^{\sigma}f_{1}^{p}=\xi^{\rho-\sigma}f_{2}^{\rho}f_{3}^{\sigma}f_{1}^{p+1}-\rho(1-\sigma)\xi^{-\rho}f_{2}^{\rho-1}f_{3}^{\sigma+1}f_{1}^{p},\\
&f_{2}f_{2}^{\rho}f_{3}^{\sigma}f_{1}^{p}=(1-\rho)f_{2}^{\rho+1}f_{3}^{\sigma}f_{1}^{p+1},\\
&[e_{1}, f_{2}^{\rho}f_{3}^{\sigma}f_{1}^{p}]=\sigma(1-\rho)(-1)^{\sigma}f_{2}^{\rho+1}f_{3}^{\sigma-1}f_{1}^{p}\xi^{h_{1}-2p+1}+[p]f_{2}^{\rho}f_{3}^{\sigma}f_{1}^{p-1}[h_{1}-p+1],\\
&e_{2}f_{2}^{\rho}f_{3}^{\sigma}f_{1}^{p}-(-1)^{\rho+\sigma}f_{2}^{\rho}f_{3}^{\sigma}f_{1}^{p}e_{2}
=\rho f_{2}^{\rho-1}f_{3}^{\sigma}f_{1}^{p}[h_{2}+p+\sigma]+\sigma(-1)^{\rho}f_{2}^{\rho}f_{3}^{\sigma-1}f_{1}^{p+1}\xi^{-h_{2}-p}
\end{align*}
where $(p, \rho, \sigma) \in \mathbb{N} \times \{0,1\} \times \{0,1\}$.
With the above relations, it is easy to check $f_{i}x_{-}w_{+}=0$.

	The fourth relation above gives us $e_{2}f_{2}f_{3}f_{1}^{\ell-1}-f_{2}f_{3}f_{1}^{\ell-1}e_{2}=f_{3}f_{1}^{\ell-1}[h_{2}+\ell]$. Because $e_{2}(v_{+} \otimes v_{+}^{'})=0$ and $[h_{2}+\ell](v_{+} \otimes v_{+}^{'})=0$, we deduce $e_{2}x_{-}w_{+}=0$. 

	The third relation gives $[e_{1}, f_{2}f_{3}f_{1}^{\ell-1}]=[\ell-1]f_{2}f_{3}f_{1}^{\ell-2}[h_{1}-\ell+2]$. Because $e_{1}(v_{+} \otimes v_{+}^{'})=0$ and $[h_{1}-\ell+2](v_{+} \otimes v_{+}^{'})=0$, we deduce $e_{1}x_{-}w_{+}=0$. 

	Consequently, we conclude that $x_{-}w_{+}$ is an $\mathcal{U}_{\xi}^{H}\mathfrak{sl}(2|1)$-invariant vector. The demonstration that the vector $x_{+}w_{-}$ is $\mathcal{U}_{\xi}^{H}\mathfrak{sl}(2|1)$-invariant is analogous using the relations obtained by applying the automorphism $\omega$ of superalgebra $\mathcal{U}_{\xi}^{H}\mathfrak{sl}(2|1)$ where $\omega(e_{i})=(-1)^{\deg e_{i}}f_{i}, \omega(f_{i})=(-1)^{\deg f_{i}}e_{i}, \omega(k_{i})=k_{i}^{-1}, \omega(h_{i})=-h_{i}, i=1,2$.

	Furthermore $\Delta x_{-}=x_{-}\otimes 1 +$ a sum of tensor products of two elements of $\mathcal{U}_{\xi}^{H}\mathfrak{sl}(2|1)$ with negative weight. Thus $\Delta x_{-}w_{+}$ contains the nonzero vector $x_{-}v_{+} \otimes v_{+}^{'}=f_{2}f_{3}f_{1}^{\ell-1}v_{+} \otimes v_{+}^{'}=w_{1,1,\ell-1} \otimes v_{+}^{'}$. We conclude that the vector $x_{-}w_{+}$ is nonzero. Similarly, the vector $x_{+}w_{-}$ is nonzero.

For step 2, we use the following results:

	The decomposition of the tensor product $V \otimes V^{*}$ is a direct sum of indecomposable modules $$V \otimes V^{*}=P_{1} \oplus ... \oplus P_{m}.$$

	The set of invariant vectors $w \in V \otimes V^{*}$ is in bijection with $\coev_{V}(\mathbb{C})$ because $\Hom_{\mathscr{C}}(\mathbb{C}, V \otimes V^{*})\cong \Hom_{\mathscr{C}}(V , V)\cong \mathbb{C}$.

	The vector $w_{+}$ (resp. $w_{-}$) is the highest weight vector (resp. lowest weight vector) of $V \otimes V^{*}$. Then there exists a unique integer $k$ (resp. $l$) such that $w_{+} \in P_{k}$ (resp. $w_{-} \in P_{l}$).
The weight of $w_{+}$ (resp. $w_{-}$) is $\lambda_{+}=(2\ell-2,-\ell)$ (resp. $\lambda_{-}=(-2\ell+2,\ell)$). Because $\lambda_{-}=-\lambda_{+}$ and $(V \otimes V^{*})^{*}\simeq (V \otimes V^{*})$, this implies $P_{k}^{*}\simeq P_{l}$.
	
	In addition, $\coev_{V}(1) \in P_{l}, \coev_{V}(1) \in P_{k}$ because $x_{+}P_{l}\subset P_{l}, x_{-}P_{k}\subset P_{k}$, then $P_{k}=P_{l}$. That is $P_{k}=P_{k}^{*}$. By Theorem 3.1.3 \cite{NgJkBp13}, it gives us the affirmation that $V_{\mu}$ ambidextrous.
\end{proof}
\begin{rmq} All typical modules are projective and ambidextrous.
\end{rmq}
\subsection{Modified traces on the projective modules}
\begin{Def}
Let $\mathcal{I}$ be an ideal of $\mathscr{C}$ {\em (see \cite{NgJkBp11})}. The family of linear applications $t =(t_{V}: \End_{\mathscr{C}}(V) \to \Bbbk)_{V \in \mathcal{I}}$ is a {\em trace (modified trace)} on $\mathcal{I}$ 
if it satisfies:\\
$\forall U,V \in \mathcal{I}, \forall W \in \mathscr{C}$,
$$\forall f \in \Hom_{\mathscr{C}}(U,V), \forall g \in \Hom_{\mathscr{C}}(V, U), {t}_{V}(f \circ g)= {t}_{U}(g \circ f)$$
$$\forall f \in \End_{\mathscr{C}}(V \otimes W), \ t_{V\otimes W}(f)=t_{V}(\ptr_{R}(f)).$$
\end{Def}
	
	We also have $$\forall f \in \End_{\mathscr{C}}(W \otimes V), \ t_{W\otimes V}(f)=t_{V}(\ptr_{L}(f)).$$ 
	
	Given $V$ as a typical module. The module $V$ is ambidextrous and projective. This implies that the ideal generated by this module is $\mathcal{I}_{V}=\texttt{Proj}$ (see \cite{NgJkBp11}). So we have the following theorem.
	\begin{The} \label{main theo}
	There exists a unique modified trace $t=\{t_{P}\}_{P \in \texttt{Proj}}$ on the ideal \texttt{Proj} of projective modules of $\mathscr{C}^{H}$, 
$$t_{P}: \End(P) \rightarrow \mathbb{C}, P \in \texttt{Proj}.$$
	\end{The}
	If $P=V_{\mu}$ is a typical module, then $t_{V_{\mu}}(f)=\langle f\rangle d(\mu), f \in \End(V_{\mu})$, $d(\mu)=t_{V_{\mu}}(\Id_{V_{\mu}})$ is determined by the Definition \ref{dn d}.
\subsection{Invariants of embedded graphs}
	Recall that $\mathscr{C}^{H}$ is the $\mathbb{C}$-linear ribbon category of the finite dimensional representations of $\mathcal{U}_{\xi}^{H}\mathfrak{sl}(2|1)$, \texttt{Proj} is the ideal of projective modules and $t$ is a trace on \texttt{Proj}.

	We call $\mathscr{G}$ the set of $\mathscr{C}^{H}$-colored closed ribbon graphs, that are the $\mathscr{C}^{H}$-colored ribbon graphs in $S^3$. We have $\mathscr{G}\cong \End_{\mathcal{T}}(\emptyset)$.

	We use the concept of a cutting presentation of $\mathscr{C}^{H}$-colored closed ribbon graph: If a diagram $T$ represents a $\mathscr{C}^{H}$-colored ribbon graph which is an endomorphism of $\mathcal{T}$, its lower and upper parts are formed by the same sequences of $k$ vertical colored strands. It is then possible, as for a braid of $k$ strands, to consider the closure $\widehat{T}$ obtained by joining its $k$ top vertices to its $k$ bottom vertices by $k$ parallel strands. This construction is actually the categorical trace in $\mathcal{T}$: we have $\widehat{T}=\tr_{\mathcal{T}}(T) \in \End_{\mathcal{T}}(\emptyset).$ We say that $T$ is a cutting presentation with $k$ strands of the closed graph $\widehat{T}$ and that $\widehat{T}$ is the closure of $T$ (see \cite{BePM12}). 

	A closed graph $T$ of $\mathcal{T}$ is said to be $\mathscr{C}^{H}$-colored {\em admissible} if there is at least one strand of $T$ colored by $P \in \texttt{Proj}$. Let $\mathscr{G}_{a}$ be the set of isotopy classes of $\mathscr{C}^{H}$-colored admissible ribbon graphs.

	From the trace $t$ on $\texttt{Proj}$ we have the theorem below.
\begin{The} \label{dl F}
The application 
\begin{align*}
F': \mathscr{G}_{a} &\rightarrow \mathbb{C}\\
\widehat{T} &\mapsto t_{P}(F(T))
\end{align*}
is well defined. Here, $P \in \texttt{Proj}, T \in \End_{\mathcal{T}}((P,+))$ is a cutting presentation with one strand of $\widehat{T}$. That is to say the complex number $t_{P}(F(T))$ does not depend on the choice of $T$ but only of the isotopy class of the $\mathscr{C}^{H}$-colored graph $\widehat{T}$.
\end{The}
\begin{proof}
	First, we select an edge of $\widehat{T}$ and cut, we have the graph $T$. Then, we select and cut a second edge of $\widehat{T}$, we have the graph $T'$. 
By cutting $\widehat{T}$ in both these places, one obtains a graph $T_{2} \in \End_{\mathscr{G}_{a}}((P,+), (P', +))$ which is a presentation with two strands of $\widehat{T}$ and such that
$T=\epsh{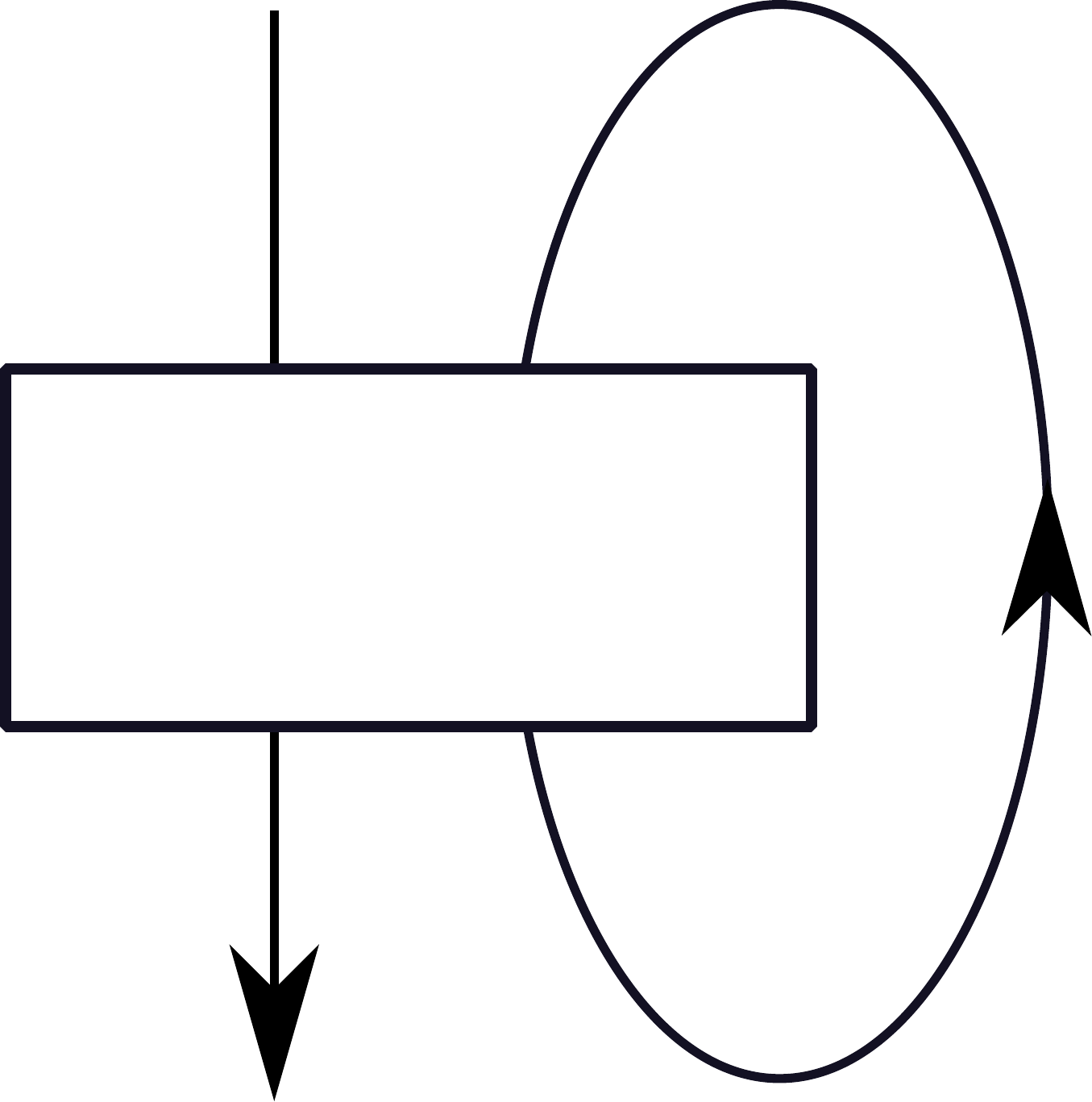}{7ex}
      \put(-23,0){\ms{T_{2}}}$, 
$T'=\epsh{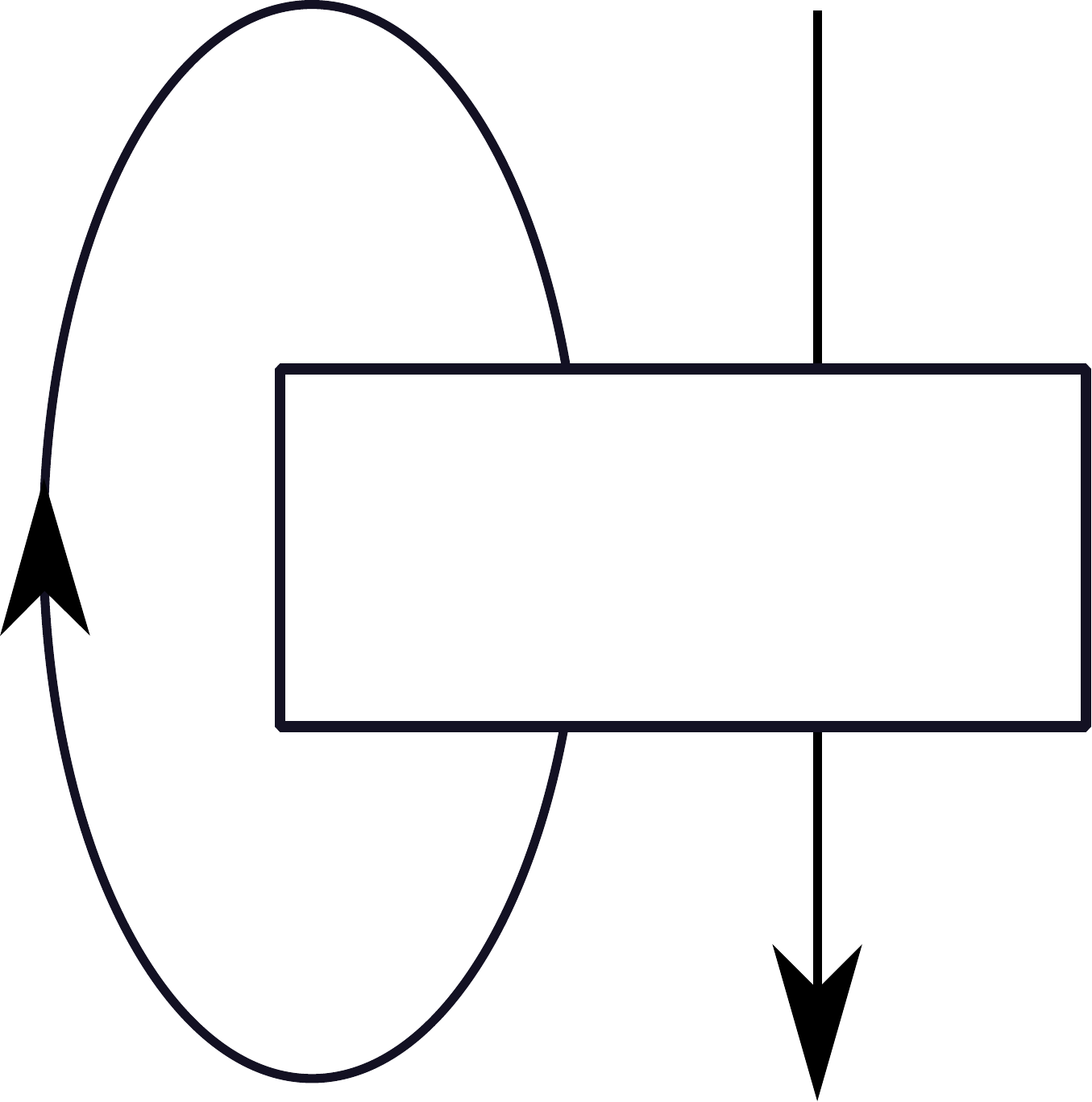}{7ex}
      \put(-19,0){\ms{T_{2}}}$.
 Finally we use the properties of the compatibility of trace  $t$:
\begin{align*}
t_{P}(F(T))&=t_{P}(\ptr_{R}(F(T_{2})))=t_{P \otimes P'}(F(T_{2}))\\
&=t_{P'}(\ptr_{L}(F(T_{2})))=t_{P'}(F(T')).
\end{align*}
\end{proof}
\begin{rmq} In the case $P=V_{\mu}$ typical, we have 
$$F'\left(\epsh{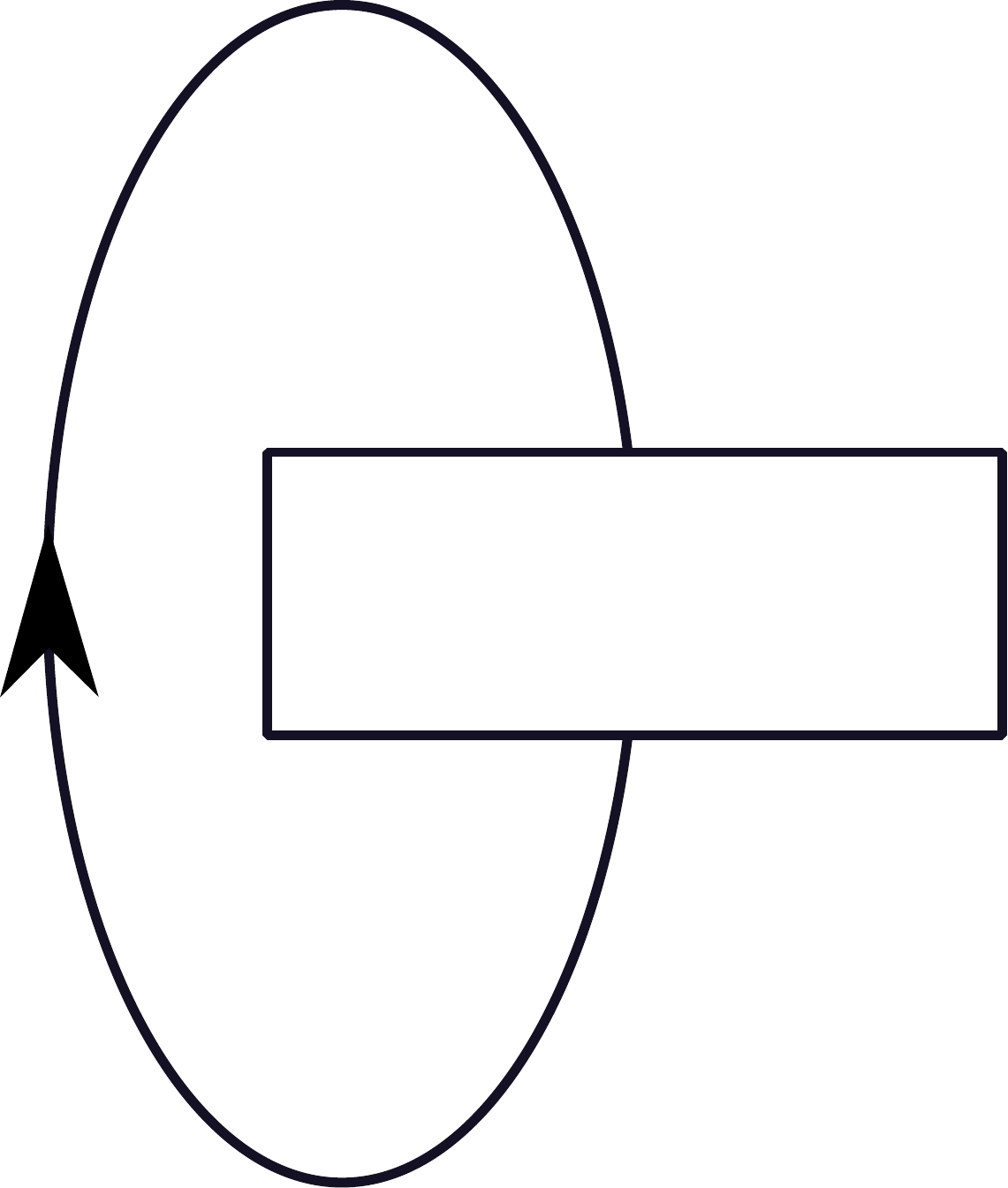}{7ex} \put(-14,0){\ms{T}}\right)
=d(\mu)\left< \epsh{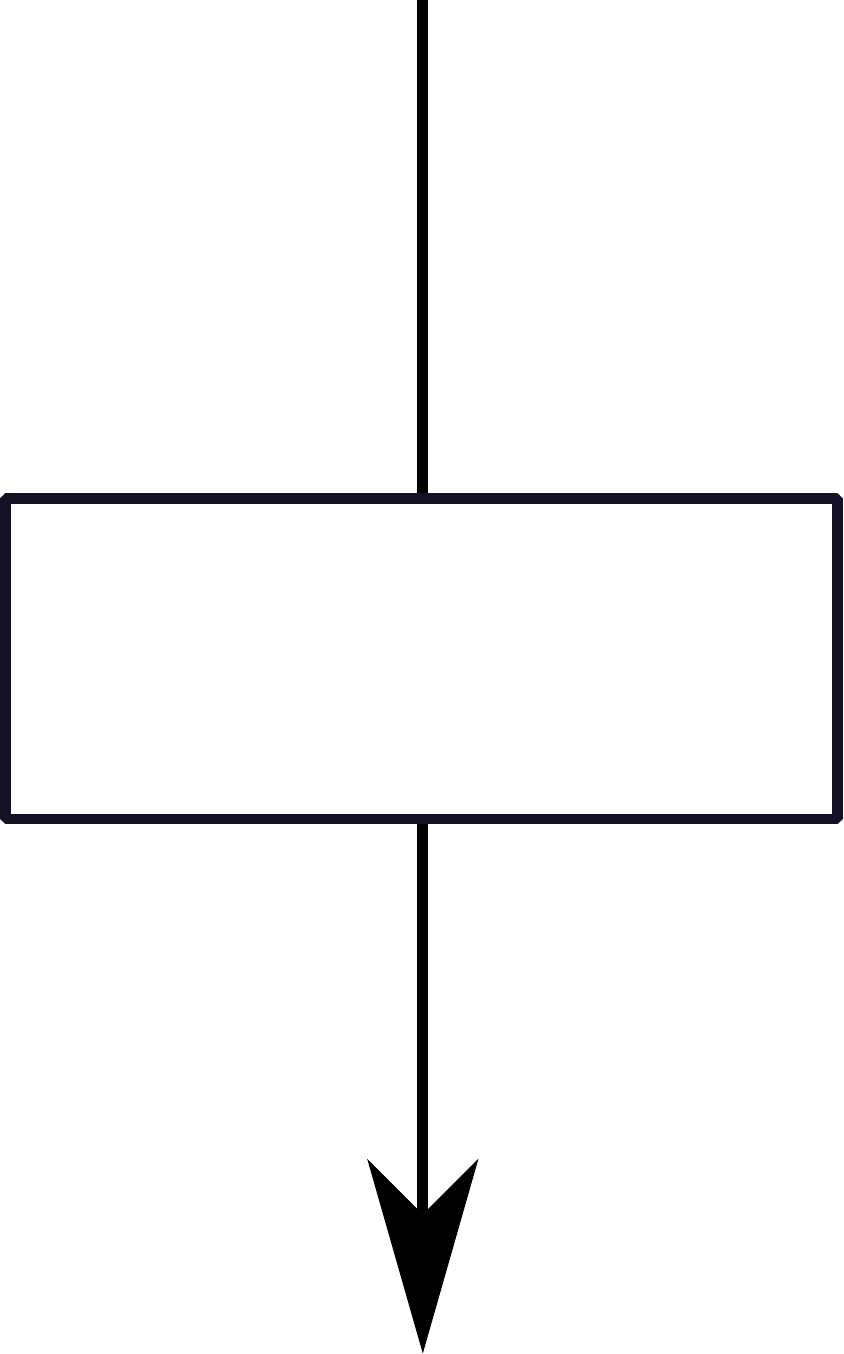}{7ex} \put(-14,0){\ms{T}} \right>.$$
\end{rmq}
	The affirmation of the above theorem gives us a link invariant in the following corollary.
\begin{HQ}
The application $F': \{\text{links}\  \mathbb{C}^{2}- \text{colored} \} \rightarrow \mathbb{C}$ enables to associate with each link with $n$ ordered components a meromorphic function on $\mathbb{C}^{2n}$.
\end{HQ}

\section{Invariant of 3-manifolds}
 In the article \cite{FcNgBp14} the authors constructed $\mathscr{C}$-decorated 3-manifold invariants where $\mathscr{C}$ is a ribbon category. In the previous section, it was proven that $\mathscr{C}^{H}$ is a ribbon category, this suggests we construct an invariant of $\mathscr{C}^{H}$-decorated 3-manifolds. We recall some concepts, definitions and results from \cite{FcNgBp14}.
\subsection{Relative $\mathit{G}$-modular categories}
  Let $\mathscr{C}$ be a $\Bbbk$-linear ribbon category where $\Bbbk$ is a field. A set of objects of $\mathscr{C}$ is said to be commutative if for any pair $\{V, W \}$ of these objects, we have $c_{V,W}\circ c_{W,V}=\Id_{W \otimes V}$ and $\theta_{V}=\Id_{V}$. Let $(Z, +)$ be a commutative group. A {\em realization} of $Z$ in $\mathscr{C}$ is a commutative set of objects $\{{\varepsilon^{t}}\}_{t \in Z}$ such that $\varepsilon^{0}=\mathbb{I}, \qdim(\varepsilon^{t})=1$ and $\varepsilon^{t} \otimes \varepsilon^{t^{'}}=\varepsilon^{t+t{'}}$ for all $t, t{'} \in Z$.

 A realization of $Z$ in $\mathscr{C}$ induces an action of $Z$ on isomorphism classes of objects of $\mathscr{C}$ by $(t, V) \mapsto \varepsilon^{t} \otimes V$. We say that $\{{\varepsilon^{t}}\}_{t \in Z}$ is a {\em free realization} of $Z$ in $\mathscr{C}$ if this action is free. This means that $\forall t \in Z \backslash \{0\}$ and for any simple object $V \in \mathscr{C}, V \otimes \varepsilon^{t} \not\simeq V$. We call {\em simple $Z$-orbit} the reunion of isomorphism classes of an orbit for this action.
 \begin{Def}[\cite{FcNgBp14}] \label{G-modrel}
 Let $(\mathit{G},\times)$ and $(Z, +)$ be two commutative groups. A $\Bbbk$-linear ribbon category $\mathscr{C}$ is $\mathit{G}$-modular relative to $\mathcal{X}$ with modified dimension {\em d} and periodicity group $Z$ if 
 \begin{enumerate}
 \item  the category $\mathscr{C}$ has a $\mathit{G}$-grading $\{\mathscr{C}_g\}_{g \in \mathit{G}}$,
 \item  the group $Z$ has a free realization $\{{\varepsilon^{t}}\}_{t \in Z}$ in $\mathscr{C}_{1}$ (where $1 \in \mathit{G}$ is the unit),
 \item  there is a $\mathbb{Z}$-bilinear application $\mathit{G}\times Z \rightarrow \Bbbk^{\times}, (g,t)\mapsto g^{\bullet t}$ such that $\forall V \in \mathscr{C}_{g}, \forall t \in Z, c_{V,\varepsilon^{t}}\circ c_{\varepsilon^{t}, V}=g^{\bullet t}\Id_{\varepsilon^{t} \otimes V}$,
 \item there exists $\mathcal{X} \subset \mathit{G}$ such that $\mathcal{X}^{-1}=\mathcal{X}$ and $\mathit{G}$ cannot be covered by a finite number of
translated copies of $\mathcal{X}$, in other words $\forall g_{1}, ..., g_{n} \in \mathit{G}, \cup_{i=1}^{n}(g_{i}\mathcal{X}) \neq \mathit{G}$,
 \item for all $g \in \mathit{G}\setminus \mathcal{X}$, the category $\mathscr{C}_{g}$ is semi-simple and its simple objects are in the reunion of a finite number of simple $Z$-orbits,
 \item  there exists a nonzero trace $t$ on ideal $\texttt{Proj}$ of projective objects of $\mathscr{C}$ and {\em d} is the associated modified dimension,
 \item  there exists an element $g \in \mathit{G}\setminus \mathcal{X}$ and an object $V \in \mathscr{C}_{g}$ such that the scalar $\Delta_{+}$ defined in Figure \ref{Kirby colour} is nonzero; similarly, there exists an element $g \in \mathit{G}\setminus \mathcal{X}$ and an object $V \in \mathscr{C}_{g}$ such that the scalar $\Delta_{-}$ defined in Figure \ref{Kirby colour} is nonzero,
 \begin{figure}
   \centering
   $$
   \begin{array}{ccc}
     F\left( \epsh{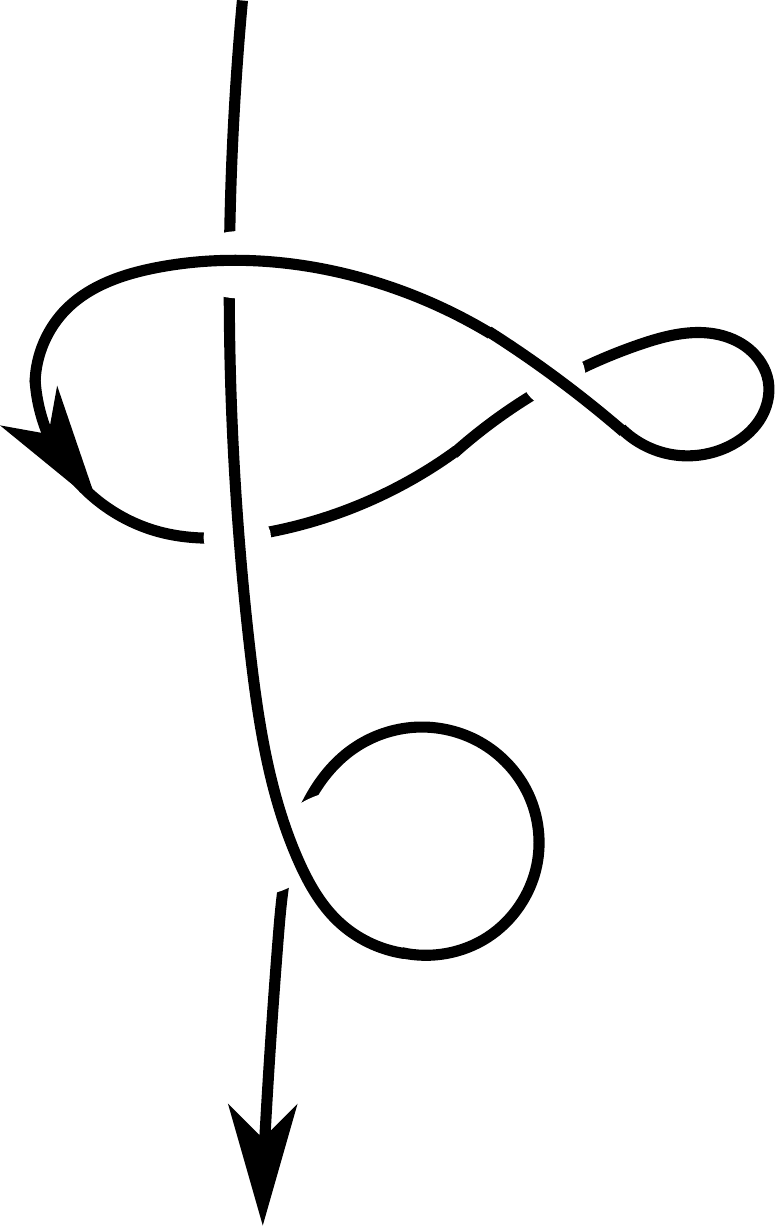}{10ex}
 		\put(-35,2){\ms{\Omega_{\overline{\mu}}}}
 		\put(-20,-22){\ms{V}}\right) = \Delta_{-}\Id_{V}, \ 
 	F\left( \epsh{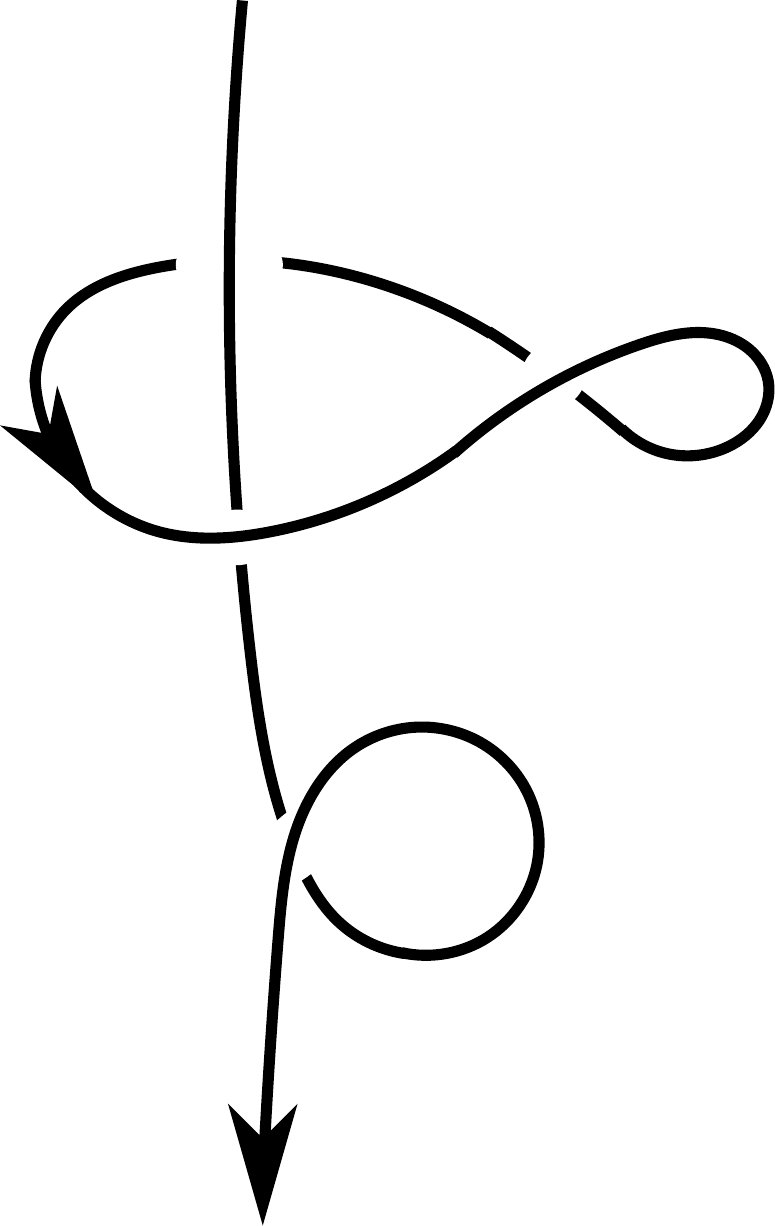}{10ex}
 		\put(-35,2){\ms{\Omega_{\overline{\mu}}}}
 		\put(-20,-22){\ms{V}}\right) = \Delta_{+}\Id_{V}
   \end{array}
   $$
	\caption{{$V \in \mathscr{C}_{g}$ and $\Omega_{\overline{\mu}}$ is a Kirby color of degree $\mu$.}}
	\label{Kirby colour}
 \end{figure} 
  \item the morphism $S(U, V)=F(H(U,V)) \neq 0 \in \End_{\mathscr{C}}(V)$, for all simple objects 	$U,V \in \texttt{Proj}$, where $$H(U,V)= \epsh{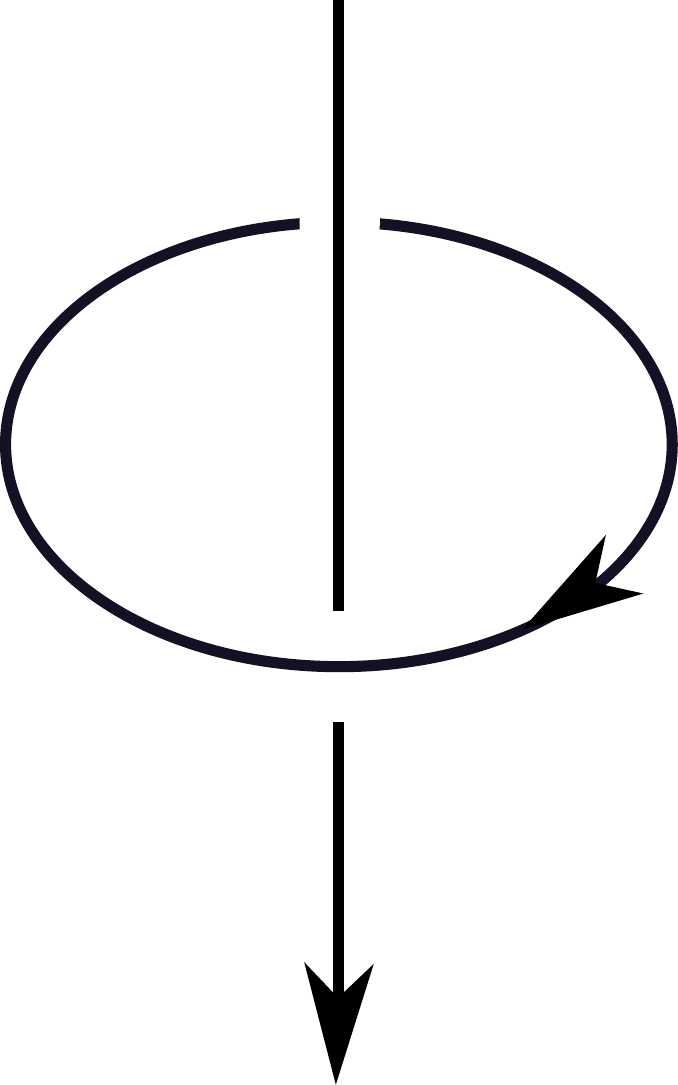}{7ex}
      \put(-6,-7){\ms{U}}
			\put(-11,18){\ms{V}} \in \End_{\mathscr{C}}((V, +)).$$
 \end{enumerate}
  \end{Def}

	The category $\mathscr{C}^{H}$ of $\mathcal{U}_{\xi}^{H}\mathfrak{sl}(2|1)$-modules is $\mathit{G}$-modular relative to $\mathcal{X}$. Indeed, we have $\mathscr{C}^{H}$ being $\mathit{G}$-graded by $\mathit{G}=\mathbb{C}/\mathbb{Z}\times \mathbb{C}/\mathbb{Z}$. We set $Z=\mathbb{Z} \times \mathbb{Z}$ and $\{{\varepsilon^{n}}\}_{n \in Z}$ the set of simple highest weight modules $n=(n_{1}\ell, n_{2}\ell)$, i.e. $\varepsilon^{n}$ is a $\mathcal{U}_{\xi}^{H}\mathfrak{sl}(2|1)$-module of dimension 1 (with the basis $\{w\}$) determined by $h_{1}w=n_{1}\ell w, h_{2}w=n_{2}\ell w, e_{i}w=f_{i}w=0$. Because $c_{\varepsilon^{m}, \varepsilon^{n}}=\tau $ and $\theta_{\varepsilon^{n}}=\Id$, the two conditions (1) and (2) of the Definition \ref{G-modrel} are satisfied.

	We consider a typical module $V_{\mu}$. We have $c_{\varepsilon^{n}, V_{\mu}}(w \otimes w_{\rho, \sigma, p})=\tau \circ \mathcal{R}(w \otimes w_{\rho, \sigma, p})=\xi^{-n_{1}\ell\mu_{2}-n_{2}\ell\mu_{1}-2n_{2}\ell\mu_{2}}w_{\rho, \sigma, p} \otimes w$. Next $c_{ V_{\mu}, \varepsilon^{n}} \circ c_{\varepsilon^{n}, V_{\mu}}(w \otimes w_{\rho, \sigma, p})=c_{ V_{\mu}, \varepsilon^{n}}(\xi^{-n_{1}\ell\mu_{2}-n_{2}\ell\mu_{1}-2n_{2}\ell\mu_{2}}w_{\rho, \sigma, p} \otimes w)=\xi^{-2n_{1}\ell\mu_{2}-2n_{2}\ell\mu_{1}-4n_{2}\ell\mu_{2}}w \otimes w_{\rho, \sigma, p}=\xi^{-2\ell(\mu_{2}n_{1}+(\mu_{1}+2\mu_{2})n_{2})}w \otimes w_{\rho, \sigma, p}.$ So we can determine the $\mathbb{Z}$-bilinear application $\mathit{G}\times Z \rightarrow \mathbb{C}^{\times}, (\overline{\mu},n) \mapsto \xi^{-2\ell(\mu_{2}n_{1}+(\mu_{1}+2\mu_{2})n_{2})}$ which satisfies $c_{ V_{\mu}, \varepsilon^{n}} \circ c_{\varepsilon^{n}, V_{\mu}}(w \otimes w_{\rho, \sigma, p})=\xi^{-2\ell(\mu_{2}n_{1}+(\mu_{1}+2\mu_{2})n_{2})}\Id_{\varepsilon^{n} \otimes V_{\mu}}(w \otimes w_{\rho, \sigma, p})$. This means that we have condition (3) of the definition. Condition (4) is also satisfied with $\mathcal{X}=\mathit{G}_{s}=\left\{\overline{0}, \overline{\frac{1}{2}}\right\}\times \mathbb{C}/ \mathbb{Z} \cup \mathbb{C}/ \mathbb{Z} \times \left \{\overline{0}, \overline{\frac{1}{2}}\right\} \cup \left\{(\overline{\mu_1}, \overline{\mu_2}): \overline{\mu_1} + \overline{\mu_2}\in\left\{\overline{0}, \overline{\frac{1}{2}}\right\}\right\}$. It was proven that $\mathscr{C}_{g}^{H}$ is semi-simple for $g \in \mathit{G} \setminus \mathit{G}_{s}$ (Theorem \ref{dl2}) and $V_{\mu} \otimes \varepsilon^{n} \simeq V_{\mu + \ell n}$, i.e. 
the condition (5) is satisfied. Theorem \ref{main theo} implies that condition (6) is true.

	To compute $\Delta_{-}$, we first use the graphical calculus
\begin{align*}
F\left( \epsh{Figure51}{10ex}
 		\put(-35,2){\ms{\Omega_{\overline{\mu}}}}
 		\put(-20,-22){\ms{V_{\mu}}}\right) 
 		&= \sum_{s,t=0}^{\ell-1}d(\mu_{st})F\left( \epsh{Figure51}{10ex}
 		\put(-38,2){\ms{V_{\mu_{st}}}}
 		\put(-20,-22){\ms{V_{\mu}}}\right)\\
 		&=\sum_{s,t=0}^{\ell-1}d(\mu_{st}) \left\langle \theta_{V_{\mu}}^{-1}\right\rangle \left\langle \theta_{V_{\mu_{st}}^{*}}^{-1}\right\rangle F \left(\epsh{Figure6}{7ex}
      \put(-11,-7){\ms{V_{\mu_{st}}}}
			\put(-11,18){\ms{V_{\mu}}} \right)\\
		&=\sum_{s,t=0}^{\ell -1}d(\mu_{st}) \left\langle \theta_{V_{\mu}}^{-1}\right\rangle \left\langle \theta_{V_{\mu_{st}}^{*}}^{-1}\right\rangle S^{'}(\mu_{st},\mu)\Id_{V_{\mu}}.
\end{align*}
	We have $$\left\langle \theta_{V_{\mu}}^{-1}\right\rangle=-\xi^{2(\alpha_{2}^{2}+\alpha_{1}\alpha_{2})}, \left\langle \theta_{V_{\mu_{st}}^{*}}^{-1}\right\rangle = -\xi^{2((\alpha_{2}+t)^{2}+(\alpha_{1}+s)(\alpha_{2}+t))}$$ and  
	$S^{'}(\mu_{st},\mu)=\xi^{-4\alpha_{2}(\alpha_{2}+t)-2(\alpha_{2}(\alpha_{1}+s)+\alpha_{1}(\alpha_{2}+t))} \dfrac{1}{\ell d(\mu)}$.\\
	 Thus  
\begin{align*}
F\left( \epsh{Figure51}{10ex}
 		\put(-35,2){\ms{\Omega_{\overline{\mu}}}}
 		\put(-20,-22){\ms{V_{\mu}}}\right)&=\sum_{s,t=0}^{\ell -1}\frac{d(\mu_{st})}{\ell d(\mu)}\xi^{2(t^{2}+st)}\Id_{V_{\mu}} \\
 		&=\frac{1}{\ell d(\mu)\{ \ell \alpha_{1} \}}\sum_{s,t=0}^{\ell -1} \dfrac{\{ \alpha_{1}+s \}}{\{ \alpha_{2}+t \}\{ \alpha_{1}+ \alpha_{2}+s +t \}}\xi^{2(t^{2}+st)}\Id_{V_{\mu}}\\
 		&=\frac{1}{\ell d(\mu)\{ \ell \alpha_{1} \}}\sum_{s,t=0}^{\ell -1} \left( \frac{\xi^{-(\alpha_{2}+t)}}{\{ \alpha_{2}+t \}} - \frac{\xi^{-(\alpha_{1}+ \alpha_{2}+s +t)}}{\{ \alpha_{1}+ \alpha_{2}+s +t \}}\right)\xi^{2(t^{2}+st)}\Id_{V_{\mu}}.
\end{align*}
Because
\begin{align*}
\sum_{s,t=0}^{\ell-1} \frac{\xi^{-(\alpha_{2}+t)}\xi^{2(t^{2}+st)}}{\{ \alpha_{2}+t \}}
	&=\sum_{t=0}^{\ell-1}\xi^{2t^{2}}\frac{\xi^{-(\alpha_{2}+t)}}{\{ \alpha_{2}+t \}}\sum_{s=0}^{\ell-1}\xi^{2st}\\
	&=\sum_{t=0}^{\ell-1}\xi^{2t^{2}}\frac{\xi^{-(\alpha_{2}+t)}}{\{ \alpha_{2}+t \}} \ell \delta_{t}^{0}\\
	&=\frac{\ell \xi^{-\alpha_{2}}}{\{ \alpha_{2} \}},
\end{align*}	
\begin{align*}
\sum_{s,t=0}^{\ell-1} \frac{\xi^{-(\alpha_{1}+ \alpha_{2}+s +t)}\xi^{2(t^{2}+st)}}{\{ \alpha_{1}+ \alpha_{2}+s +t \}}
	&=-\sum_{s,t=0}^{\ell-1}\xi^{2(t^{2}+st)} \frac{1}{1-\xi^{2(\alpha_{1}+ \alpha_{2}+s +t)}} \\
	&=-\sum_{s,t=0}^{\ell-1}\xi^{2(t^{2}+st)} \sum_{k=0}^{\infty}\xi^{2k(\alpha_{1}+ \alpha_{2}+s +t)} \\
	&=-\sum_{k=0}^{\infty}\sum_{t=0}^{\ell-1}\xi^{2(t^{2}+k\alpha_{1}+k\alpha_{2}+kt)}\sum_{s=0}^{\ell-1}\xi^{2(k+t)s} \\
	&=-\sum_{k=0}^{\infty}\sum_{t=0}^{\ell-1}\xi^{2(t^{2}+k\alpha_{1}+k\alpha_{2}+kt)} \ell \delta_{t+k \ \text{mod} \ \ell\mathbb{N}}^{0}\\
	&=-\ell \left(1+ \sum_{t=0}^{\ell-1}\xi^{2t^{2}}\sum_{j=1}^{\infty}\xi^{2(\ell j-t)(\alpha_{1}+\alpha_{2}+t)}\right)\\
	&=-\ell \left(1+\sum_{t=0}^{\ell -1}\xi^{-2t(\alpha_{1}+\alpha_{2})} \frac{\xi^{2\ell (\alpha_{1}+\alpha_{2})}}{1-\xi^{2\ell (\alpha_{1}+\alpha_{2})}}\right) \\
		&=-\ell + \dfrac{\ell \xi^{\alpha_{1}+\alpha_{2}}}{\{\alpha_{1}+\alpha_{2}\}}
\end{align*}
then
\begin{align*}
F\left( \epsh{Figure51}{10ex}
 		\put(-35,2){\ms{\Omega_{\overline{\mu}}}}
 		\put(-20,-22){\ms{V_{\mu}}}\right)
 		&=\dfrac{1}{d(\mu)\{ \ell \alpha_{1} \}} \left(\dfrac{1}{\xi^{\alpha_{2}}\{ \alpha_{2} \}}-\dfrac{\xi^{\alpha_{1}+\alpha_{2}}}{\{\alpha_{1}+\alpha_{2}\}} +1 \right)\Id_{V_{\mu}}\\
 		&=\dfrac{1}{\{\alpha_{1}\}} \left(\{\alpha_{1}+\alpha_{2}\}\xi^{-\alpha_{2}}-\{\alpha_{2}\}\xi^{\alpha_{1}+\alpha_{2}}+\{\alpha_{2}\}\{\alpha_{1}+\alpha_{2}\}\right)\Id_{V_{\mu}} \\
 		&=\dfrac{1}{\{\alpha_{1}\}}\{\alpha_{1}\}=\Id_{V_{\mu}}.
\end{align*}
This means that $\Delta_{-}=1$.

By using the automorphism $\omega$ of superalgebra $\mathcal{U}_{\xi}\mathfrak{sl}(2|1)$ where $\omega(e_{i})=(-1)^{\deg e_{i}}f_{i}, \omega(f_{i})=(-1)^{\deg f_{i}}e_{i}, \omega(k_{i})=k_{i}^{-1}, \omega(h_{i})=-h_{i}, i=1,2$ and computing we also have $\Delta_{+}=1$.
 Condition (8) is obviously true.

Hence category $\mathscr{C}^{H}$ is relatively $\mathit{G}$-modular.
\subsection{Invariants of 3-manifolds}
\begin{Def}
Let $(M, T, \omega)$ be a triple where $M$ is a compact connected oriented $3$-manifold, $T \subset M$ is a $\mathscr{C}^{H}$-colored ribbon graph (possibly empty) and $\omega \in H^{1}(M \setminus T, \mathit{G})$.
\begin{enumerate}
\item The triple $(M, T, \omega)$ is {\em compatible} if each edge $e$ of $T$ is colored by an element of $\mathscr{C}_{\omega(m_{e})}$ where $m_{e}$ is an oriented meridian of the edge $e$.
\item Let $L \cup T \subset S^{3}$ where $L$ is an oriented link in $S^{3} \setminus T$ which gives a presentation of $(M,T)$ by surgery. The presentation $L \cup T$ is {\em computable} if for each component $L_{i}$ of $L$ 
whose meridian is denoted $m_{i}$, we have $\omega(m_{i}) \notin \mathcal{X}$.
\end{enumerate}
\end{Def}

We suppose that $(M, T, \omega)$ is a compatible triple.
\begin{Def}
The formal linear combination $\Omega_{\overline{\mu}}=\sum_{\mu_{i} \in \overline{\mu}}d(V_{\mu_{i}})V_{\mu_{i}}$ is a Kirby color of degree $\overline{\mu} \in \mathit{G} \setminus \mathit{G}_{s}$ if $\{V_{\mu_{i}} \}$ is a set of representatives of simple $Z$-orbits of $\mathscr{C}_{\overline{\mu}}$.
\end{Def}
\begin{The}
Let $(M, T, \omega)$ a compatible triple admitting a computable presentation $L \cup T \subset S^{3}$ then
$$ N(M, T, \omega)=F^{'}(L_{\omega} \cup T)$$
is a well defined topological invariant, i.e. 
depends only on the diffeomorphism class of the triple $(M, T, \omega)$ where $L_{\omega}$ is obtained as the link $L$ in which we have colored the $i$-th component $L_{i}$ by a Kirby color of degree $\omega(m_{i})$ where $m_{i}$ is a meridian of $L_{i}$.
\end{The}
\subsection{Example} 
We consider an example in the case $\ell = 3$. Let $M$ be the lens space $L(5,2)$ which is given by surgery presentation on the Hopf link $L$ (Figure \ref{Hopf 32}). It has two oriented components  $L_i, i=1,2$ with framings $3, 2$ and let $m_i$ be an oriented meridian of $L_i$. 
\begin{figure}
   \centering
   $$
   \begin{array}{ccc}
      \epsh{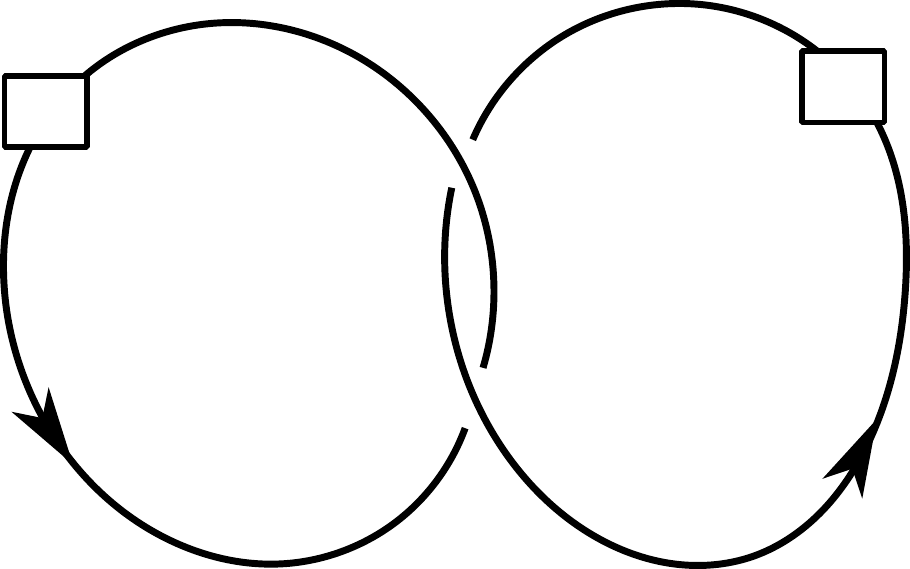}{12ex}
 		\put(-102,-22){\ms{\Omega_{\overline{\mu}}}}
 		\put(-2,-22){\ms{\Omega_{\overline{\mu}^{'}}}}
 		\put(-96,19){\ms{3}}
 		\put(-9,22){\ms{2}}
   \end{array}
   $$
	\caption{}
	\label{Hopf 32}
 \end{figure} 
The linking matrix of $L$ with respect to the components $L_i$ is 
$$\lk=\begin{pmatrix} 3&1 \\ 1&2 \end{pmatrix}.$$

Let $\omega \in H^{1}(M \setminus T, \mathit{G})$ and suppose that the triple $(M, \emptyset, \omega)$ is computable. We compute the values $\omega=(\omega^1, \omega^2)$ where $\overline{\mu}=\omega^1=\omega(m_{1}), \overline{\mu}^{'}=\omega^2=\omega(m_{2})$ from the equations $3\overline{\mu}+\overline{\mu}^{'}=0$ and $\overline{\mu}+2\overline{\mu}^{'}=0$ (in $\mathbb{C}/\mathbb{Z} \times \mathbb{C}/\mathbb{Z}$). Hence $\overline{\mu}=(\frac{k}{5}, \frac{2k}{5}), \overline{\mu}^{'}=(\frac{2k}{5}, \frac{4k}{5}), k=1,...,4$. Here we set $\omega_k=(\omega_{k}^1, \omega_{k}^2),\  \omega_{k}^1=(\frac{k}{5}, \frac{2k}{5}),\  \omega_{k}^2=(\frac{2k}{5}, \frac{4k}{5}), k=1,...,4$. We have $\omega_{4}=-\omega_{1},\  \omega_{3}=-\omega_{2}$. Using variables as in Lemma \ref{kihieu alpha} we have $(\alpha_{1}, \alpha_{2})=\overline{\mu}+(-\ell+1, \frac{\ell}{2})=(\frac{k}{5}-2, \frac{2k}{5}+\frac{3}{2}), (\alpha_{1}^{'}, \alpha_{2}^{'})=\overline{\mu}^{'}+(-\ell+1, \frac{\ell}{2})=(\frac{2k}{5}-2, \frac{4k}{5}+\frac{3}{2})$.

We color the $i$-th component $L_i$ by a Kirby color of degree $\omega(m_{i})$, i.e. $\Omega_{\omega(m_{1})}=\Omega_{\overline{\mu}}=\sum_{s,t=0}^{2}d(\alpha_{st})V_{\alpha_{st}}$ and $\Omega_{\omega(m_{2})}=\Omega_{\overline{\mu}^{'}}=\sum_{i,j=0}^{2}d(\alpha_{ij}^{'})V_{\alpha_{ij}^{'}}$ where $\alpha_{st}=(\alpha_{1}+s, \alpha_{2}+t), \alpha_{ij}^{'}=(\alpha_{1}^{'} +i, \alpha_{2}^{'}+j)$. By Lemma \ref{kihieu alpha}, Proposition \ref{tinh S'} we have
$$N(M, \emptyset, \omega)=\sum_{s,t}\sum_{i,j}d(\alpha_{st}) d(\alpha_{ij}^{'})\left<\theta_{V_{\alpha_{st}}} \right>^{3} \left<\theta_{V_{\alpha_{ij}^{'}}} \right>^{2} d(\alpha_{st}) S^{'}(\alpha_{ij}^{'}, \alpha_{st})$$
in which 
\begin{align*}
d(\alpha_{st})&=\frac{\{\alpha_1+s\}}{\ell\{\ell(\alpha_1+s)\}\{\alpha_2+t\}\{\alpha_1+\alpha_2+s+t\}},\\
\left<\theta_{V_{\alpha_{st}}} \right> &=-\xi^{-2\left( (\alpha_2+t)^{2}+(\alpha_1+s)(\alpha_2+t)\right)},\\
\left<\theta_{V_{\alpha_{ij}^{'}}} \right> &=-\xi^{-2\left( (\alpha_2^{'}+j)^{2}+(\alpha_1^{'}+i)(\alpha_2^{'}+j)\right)},\\
S^{'}(\alpha_{ij}^{'}, \alpha_{st}) &=\frac{1}{\ell d(\alpha_{st})}\xi^{-4(\alpha_2^{'}+j)(\alpha_2+t)-2\left( (\alpha_2^{'}+j)(\alpha_{1}+s)+(\alpha_1^{'}+i)(\alpha_2+t)\right)}.
\end{align*}
Using computer algebra software Sagemath, we have ($\xi^{\frac{1}{10}}$ has degree $8$ over $\mathbb{Q}$)
\begin{align*}
N(M, \emptyset, \pm\omega_1)&=\frac{1}{15}\left(-2\xi^{\frac{7}{10}}-2\xi^{\frac{3}{5}}-2\xi^{\frac{1}{2}}+2\xi^{\frac{2}{5}}+5\xi^{\frac{3}{10}}+2\xi^{\frac{1}{10}}\right),\\
N(M, \emptyset, \pm\omega_2)&=\frac{1}{15}\left(-7\xi^{\frac{7}{10}}-2\xi^{\frac{3}{5}}+4\xi^{\frac{1}{2}}+4\xi^{\frac{2}{5}}+2\xi^{\frac{3}{10}}+5\xi^{\frac{1}{10}}-4\right).
\end{align*}
In this case, the result $N(M, \emptyset, \omega)=N(M, \emptyset, -\omega)$ is consistent with $(M, \emptyset, \omega)\simeq (M, \emptyset, -\omega)$.

\bibliographystyle{plain} 
\bibliography{TKhao}
\end{document}